\documentclass[11pt]{amsart}

\usepackage{geometry} 
\usepackage{amsmath}
\usepackage{amssymb}
\usepackage{stmaryrd}
\usepackage[dvipsnames]{xcolor}
\usepackage{commath}
\usepackage{comment}
\usepackage{tikz-cd}
\usepackage{longtable} 
\usepackage{pdflscape} 
\usepackage{booktabs}
\usepackage{hyperref}
\definecolor{darkgoldenrod}{rgb}{0.72, 0.53, 0.04}
\definecolor{vegasgold}{rgb}{0.77, 0.7, 0.35}
\definecolor{gold(metallic)}{rgb}{0.83, 0.69, 0.22}
\definecolor{sepia}{rgb}{0.44, 0.26, 0.08}
\definecolor{silver}{rgb}{0.75, 0.75, 0.75}
\hypersetup{
 colorlinks=true,
 linkcolor=gold(metallic),
 filecolor=sepia,      
 urlcolor=vegasgold,
 citecolor=darkgoldenrod,
 pdftitle={BKR: Rank Jumps and Growth of Shafarevich--Tate groups},
    }
\usepackage{enumerate}
\usepackage[utf8]{inputenc}
\usepackage[OT2,T1]{fontenc}
\usepackage{color}

\DeclareSymbolFont{cyrletters}{OT2}{wncyr}{m}{n}
\DeclareMathSymbol{\Sha}{\mathalpha}{cyrletters}{"58}

\newcommand\Q{{\mathbb Q}}
\newcommand\C{{\mathbb C}}
\newcommand\Z{{\mathbb Z}}
\newcommand\F{{\mathbb F}}
\newcommand\cF{{\mathcal F}}
\newcommand\cE{{\mathcal E}}
\newcommand{\op}[1]{\operatorname{#1}}

\newcommand{\Selp}{\Sel_{p^{\infty}}(E/\Q_{\cyc})}

\DeclareMathOperator{\corank}{corank}
\DeclareMathOperator{\trace}{trace}
\DeclareMathOperator{\ns}{ns}
\DeclareMathOperator{\li}{li}

\DeclareMathOperator{\rank}{rank}

\DeclareMathOperator{\Sel}{Sel}
\DeclareMathOperator{\cond}{cond}

\DeclareMathOperator{\Gal}{Gal}

\DeclareMathOperator{\GL}{GL}

\DeclareMathOperator{\Frob}{Frob}
\DeclareMathOperator{\image}{Image}

\DeclareMathOperator{\cyc}{cyc}
\DeclareMathOperator{\ord}{ord}
\DeclareMathOperator{\super}{ss}
\DeclareMathOperator{\spl}{split}
\DeclareMathOperator{\tors}{tors}
\DeclareMathOperator{\height}{height}
\DeclareMathOperator{\Id}{Id}

\newtheorem*{Theorem*}{Theorem}
\newtheorem{Th}{Theorem}[section]
\newtheorem{Lemma}[Th]{Lemma}
\newtheorem*{Ques*}{Question}
\newtheorem{Prop}[Th]{Proposition}

\newtheorem{Cor}[Th]{Corollary}

\newtheorem*{Conj*}{Conjecture}
\newtheorem{lthm}{Theorem}
\newtheorem{lcor}{Corollary}

\theoremstyle{remark}
\newtheorem*{Rem*}{Remark}
\theoremstyle{definition}
\newtheorem{Defi}[Th]{Definition}
\newcommand\mtx[4] { \left( {\begin{array}{cc}
   #1 & #2 \\
   #3 & #4 \\
  \end{array} } \right)}

\makeatletter
\@namedef{subjclassname@1991}{\emph{2020} Mathematics Subject Classification}
\makeatother

\begin{document}
\title[Growth Questions in $\Z/p\Z$-Extensions]{Rank Jumps and Growth of Shafarevich--Tate Groups for Elliptic Curves in $\Z/p\Z$-Extensions}

\author[L.~Beneish]{Lea Beneish}
\address[Beneish]{Department of Mathematics, University of California, Berkeley\\ 
970 Evans Hall\\
Berkeley, CA 94720, USA.}
\email{leabeneish@berkeley.edu}

\author[D.~Kundu]{Debanjana Kundu}
\address[Kundu]{222 College Street \\ Fields Institute \\ ON, M5T 3J1, Canada.}
\email{dkundu@math.toronto.edu}

\author[A.~Ray]{Anwesh Ray}
\address[Ray]{Centre de recherches math{\'e}matiques,
Universit{\'e} de Montr{\'e}al,
Pavillon Andr{\'e}-Aisenstadt,
2920 Chemin de la tour,
Montr{\'e}al (Qu{\'e}bec) H3T 1J4, Canada}
\email{anwesh.ray@umontreal.ca}


\date{\today}

\begin{abstract}
Let $p$ be a prime.
In this paper, we use techniques from Iwasawa theory to study questions about rank jump of elliptic curves in cyclic extensions of degree $p$.
We also study growth of the $p$-primary Selmer group and the Shafarevich--Tate group in cyclic degree-$p$ extensions and improve upon previously known results in this direction.
\end{abstract}

\maketitle

\section{Introduction}
A fundamental result in the theory of elliptic curves is the \emph{Mordell--Weil Theorem}.
It states that given an elliptic curve $E$ defined over a number field $F$, its $F$-rational points form a finitely generated abelian group, i.e.,
\[
E(F) \simeq \Z^r \oplus E(F)_{\tors}
\]
where $r$ is a non-negative integer called the \emph{rank} and $E(F)_{\tors}$ is a finite group, called the \emph{torsion subgroup}.
Over $\Q$, the possible structures of $E(\Q)_{\tors}$ is known by the work of B.~Mazur (see \cite{Maz77, Maz78}).
These techniques have been extended by S.~Kamienny--M.~Kenku--F.~Momose (see for example \cite{KM88, Kam92}) to provide the complete classification of torsion subgroups for quadratic fields.
More recently, in a series of work by several authors the classification of torsion subgroups for cubic fields has been completed (see \cite{ jeon2004torsion, jeon2011families, wang2015torsion, najman2016torsion, bruin2016criterion, derickx2018torsion, derickx2021sporadic}).
In \cite{Mer96}, L.~Merel proved that for elliptic curves over any number field, the bound on the order of the torsion subgroup depends only on the degree of the number field.
Despite remarkable advances made towards understanding the torsion subgroup, the rank remains mysterious and to-date there is no known algorithm to compute it.
Almost always, obtaining information about the rank of the elliptic curve involves studying the \emph{Selmer group}.

Let $p$ be a prime number.
In \cite{Maz72}, Mazur initiated the study of $p$-primary Selmer groups of elliptic curves in $\Z_p$-extensions.
Since then, Iwasawa theory of elliptic curves has been successfully used by many authors to study rank growth in towers of number fields.
However, the scope of this paper is different.
We use results from Iwasawa theory of Selmer groups of elliptic curves to obtain results on rank growth in cyclic degree-$p$ extensions. 
Let $L/\Q$ be a finite extension and $E$ an elliptic curve over $\Q$.
Mazur and Rubin \cite{MR18} define $E$ to be \emph{diophantine-stable} in $L$ if $E(L)=E(\Q)$.
This property is of significant importance and has applications to Hilbert's $10^{\text{th}}$ problem for number fields.
In this paper, we answer the following two questions about growth of Selmer groups in cyclic degree-$p$ extensions.
\begin{enumerate}
\item Given an elliptic curve $E_{/\Q}$ with trivial $p$-primary Selmer group, for what proportion of degree-$p$ cyclic extensions does the $p$-primary Selmer group remain trivial upon base-change?
\item Given a prime $p\neq 2,3$, varying over all elliptic curves defined over $\Q$, for what proportion of elliptic curves does there exist \emph{at least one} $\Z/p\Z$-extension where the $p$-primary Selmer group remains trivial upon base-change?
\end{enumerate}
The proportion of elliptic curves is computed with respect to height (see \eqref{eqn: height defi}).
Our results are proven for elliptic curves $E_{/\Q}$ with Mordell-Weil rank $0$.
Using standard arguments from Iwasawa theory, one can show that the $p$-primary Selmer group has rank $0$ over the cyclotomic $\Z_p$-extension of a number field $L$ (in particular, over $L$ itself) if and only if the associated $\mu$-invariant and $\lambda$-invariant are trivial (see for example \cite[Corollary 3.6]{KR21}).
Controlling the $\mu$-invariant is relatively easy and it is known to behave well in $p$-extensions (see Theorem \ref{thm: Kida formula}).
So the key idea involves showing how often the $\lambda$-invariant remains trivial upon base-change to $L/\Q$. 
Given $n\geq 0$, let $L_n$ denote the cyclotomic $\Z/p^n\Z$ extension of $L$.
This is the unique $\Z/p^n\Z$-extension of $L$ contained inside $L(\mu_{p^\infty})$.
We note that we have at this point in our notation suppressed the dependence on the prime $p$.
Since $\lambda_p(E/L)\geq \rank_{\Z}(E(L_n))$ (see Lemma \ref{lambda geq rank}), proving triviality of the $\lambda$-invariant upon base-change implies the rank does not change in $L$.
In fact, we prove a stronger result, and show that the rank does not in fact increase in $L_n$ for all $n\geq 0$.
To answer the first question, we prove the following result.

\begin{lthm}\label{thm A}
Let $E_{/\Q}$ be a non-CM elliptic curve and $p$ be a fixed prime number $\geq 5$ such that the residual representation at $p$ is surjective.
Further suppose that $\mu_p(E/\Q) = \lambda_p(E/\Q) =0$.
Then, there are infinitely many $\Z/p\Z$-extensions of $\Q$ in which the $\lambda$-invariant does not increase.
In particular, in infinitely many cyclic degree-$p$ number fields $L/\Q$ the rank does not grow in $L_n$ for all $n\geq 0$, i.e., 
\[
\rank_{\Z} E(\Q) = \rank_{\Z} E(L_n) \text{ for all } n\geq 0.
\]
\end{lthm}

In the case that $E_{/\Q}$ is a CM elliptic curve, we can prove a similar result under an additional independence hypothesis which we make precise in \ref{assmpn: enemy primes for CM}. 

The condition $\lambda_p(E/\Q)=0$ can only be satisfied for elliptic curves $E_{/\Q}$ of Mordell-Weil rank $0$, and thus the above result shows that the rank remains $0$ in the fields $L_n$.

Even though our proof suggests that the $\lambda$-invariant does not jump in \emph{many} $\Z/p\Z$-extensions, we are unable to show that this is true for a positive proportion of $\Z/p\Z$-extensions.
It is known by the work of R.~Greenberg \cite[Theorem 5.1]{Gre99} that given a rank 0 elliptic curve $E$ over $\Q$, for \emph{density one} good ordinary primes $\mu_p(E/\Q) =\lambda_p(E/\Q)=0$.

In a recent paper (see \cite{GJN20}), E.~Gonz{\'a}lez-Jim{\'e}nez and F.~Najman investigated the question of when the torsion group does not grow upon base-change (see Theorem \ref{result from GJN} for the precise statement).
This, when combined with the above theorem, allows us to comment on when the Mordell--Weil group \emph{does not} grow in $\Z/p\Z$-extensions.
More precisely,
\begin{lcor}
Given a non-CM elliptic curve $E_{/\Q}$ and a prime $p>7$ such that the residual representation at $p$ is surjective and $\mu_p(E/\Q) = \lambda_p(E/\Q) =0$, there are infinitely many $\Z/p\Z$-extensions $L/\Q$ such that 
\[
E(L)=E(\Q).
\].
\end{lcor}
The same assertion holds for elliptic curves with CM provided \ref{assmpn: enemy primes for CM} holds.

A natural extension of the previous question is the following: when does the $p$-primary Selmer group grow upon base-change?
We address this question in Section \ref{Appendix}.
First, we prove a result regarding the growth of the $p$-primary Selmer group of an elliptic curve $E_{/\Q}$ upon base-change up a $\Z/p\Z$-extension (see Proposition \ref{LKR}).
This result gives a criterion for either the rank to jump, or the order of the Shafarevich--Tate group to increase upon base-change.
It applies to all primes $p\geq5$ and the method relies on exploiting the relationship between Iwasawa invariants and the Euler characteristic.
It is motivated by ideas from \cite{Raysujatha1, Raysujatha2}, as well as the use of Kida's formula (Theorem \ref{thm: Kida formula}).
This criterion is then used to show that the Selmer group becomes non-zero in a large family of $\Z/p\Z$-extensions.
Given an elliptic curve $E_{/\Q}$ and a prime $p$, we set $E[p]$ to denote the $p$-torsion subgroup of $E(\bar{\Q})$. The residual representation at $p$ refers to the Galois representation
\[\bar{\rho}:\Gal(\bar{\Q}/\Q)\rightarrow \GL_2(\F_p)\]
on $E[p]$.
More precisely, we can prove the following result.
\begin{lthm}
\label{Anwesh's thm}
Let $p\geq 5$ be a fixed prime and $E_{/\Q}$ be an elliptic curve with good ordinary reduction at $p$.
Suppose that the image of the residual representation is surjective, the $p$-primary Selmer group over $\Q$ is trivial, and $\mu_p(E/\Q)=\lambda_p(E/\Q)=0$.
Then, there is a set of primes of the form $q\equiv 1\pmod p$ with density at least $\frac{p}{(p-1)^2(p+1)}$ such that the $p$-primary Selmer group becomes non-trivial in the unique $\Z/p\Z$-extension contained in $\Q(\mu_q)$.
\end{lthm}

However, we remark that this method is unable to distinguish between jumps in rank and jumps in the order of the Shafarevich--Tate group.
These methods are currently being refined by the third named author, and in a subsequent paper, will be applied to a large number of problems in Diophantine stability and arithmetic statistics.

In response to the second question, we show the following.

\begin{lthm}
Suppose that the Shafarevich--Tate group is finite for all rank 0 elliptic curves defined over $\Q$ and \ref{assmpn: for rank 0 indep} holds.
For a positive proportion of rank 0 elliptic curves defined over $\Q$, there is \emph{at least one} $\Z/p\Z$-extension over $\Q$ disjoint from the cyclotomic $\Z_p$-extension, $\Q_{\cyc}/\Q$, with trivial $p$-primary Selmer group upon base-change.
\end{lthm}

\ref{assmpn: for rank 0 indep} is the assumption that the proportion of rank 0 elliptic curves (ordered by height) with a fixed reduction type at a prime $q$ is the same as the proportion of all elliptic curves (ordered by height) with the same property.
In other words, the reduction type at $q$ is independent of the rank of the elliptic curve.

We remark that the main reason for us to restrict the study to elliptic curves at primes of good ordinary reduction is to ensure that we can use results on $\lambda$-invariants of elliptic curves from \cite{HM99}.
There are recent results which extend the aforementioned theorem to the non-ordinary case.
It seems reasonable to expect that our results should extend to the supersingular case using the work of J.~Hatley and A.~Lei in \cite[Theorem 6.7]{HL19}.

In \cite{Ces17}, K.~{\v{C}}esnavi{\v{c}}ius showed that the $p$-Selmer group of elliptic curves over a number field becomes arbitrarily large when varying over $\Z/p\Z$-extensions.
This result is not surprising, as it is widely believed that the growth of ideal class groups and Selmer groups of elliptic curves are often analogous.
The unboundedness of the $2$-part of the ideal class group in quadratic extensions goes back to Gauss; for its generalization to an odd prime see \cite[VII-12, Theorem 4]{BCHIS66}.
The \emph{rank boundedness conjecture} for elliptic curves asks whether there is an upper bound for the Mordell--Weil rank of elliptic curves over number fields.
This question is widely open and experts are unable to come to a consensus on what to expect, see \cite[Section 3]{PPVW19} for a historical survey.
Those in favour of unboundedness argue that this phenomenon provably occurs in other global fields, and that the proven lower bound for this upper bound increases every few years.
For instance, N.~Elkies discovered an elliptic curve over $\Q$ with Mordell--Weil rank at least 28.
On the other hand, a recent series of papers by B.~Poonen et. al. provides a justified heuristic inspired by ideas from arithmetic statistics which suggests otherwise (see for example, \cite{Poo18}).

The interplay between the rank, Selmer group, and the Shafarevich--Tate group raises the question of producing elliptic curves with ``large Shafarevich--Tate groups.''
More precisely, given a number field $F$, a prime $p$, and a positive integer $n$, does there exist an elliptic curve $E_{/F}$ whose Shafarevich--Tate group contains at least $n$ elements of order $p$?
A dual question one can ask is the following: given $E_{/\Q}$, a positive integer $n$, and a prime $p$, does there exist a number field $F/\Q$ (with additional properties) such that the Shafarevich--Tate group, denoted by $\Sha(E/F)$, contains at least $n$ elements of order $p$?

In the final section, we study growth questions for Shafarevich--Tate groups.
When $p=2$, K.~Matsuno showed that the 2-rank of the Shafarevich--Tate group becomes arbitrarily large in quadratic extensions of number fields (see \cite{Mat09}).
In Theorem \ref{thm: effective matsuno unconditional} we prove an effective version of this theorem for elliptic curves without any exceptional primes.
In particular, given $n>0$ we find an upper bound on the minimal conductor of the quadratic extension (say $K$) such that $\rank_2 \Sha(E/K) := \dim_{\mathbb{F}_2}\Sha(E/K)[2]>n$.
Recently it has been shown in \cite{MP21} that there are infinitely many elliptic curves with large $\rank_2 \Sha(E/K)$.
When $p\neq 2$, P.~Clark and S.~Sharif showed that varying over all degree-$p$ extensions of $\Q$, not necessarily Galois, the $p$-rank of the Shafarevich--Tate group can become arbitrarily large, see \cite{CS10}.
They further raised a question as to whether this result is the best possible.
In Section \ref{subsection: arbitrarily large Sha in odd p case}, we show that if a conjecture of C.~David--J.~Fearnley--H.~Kisilevsky is true (see \S \ref{subsection: conjectures} for the precise statement), then the result of Clark--Sharif can be improved.
In particular, instead of varying over all degree-$p$ extensions, it suffices to vary over all $\Z/p\Z$-extensions of $\Q$.

\emph{Organization:} Including this Introduction, the article has 6 sections.
Section \ref{section: preliminaries} is preliminary in nature; after introducing the main objects of interest, we record relevant results from Iwasawa theory.
We also mention some conjectures and heuristics on rank growth of elliptic curves in Galois extensions.
In Section \ref{section: no rank jump}, we use techniques from Iwasawa theory to prove results on rank jump of elliptic curves in cyclic degree-$p$ extensions.
In Section \ref{Appendix}, we study arithmetic statistics related questions pertaining to the growth of the $p$-primary Selmer groups in $\Z/p\Z$-extensions.
In Section \ref{section: growth of Sha}, we study the growth of the $p$-rank of Shafarevich--Tate group in cyclic degree-$p$ extensions.
One of our results on rank growth requires a mild hypothesis, we provide computational evidence for the same.
The table is included in Section \ref{section: Tables}.

\section{Preliminaries}
\label{section: preliminaries}

Let $F$ be a number field and $E_{/F}$ be an elliptic curve defined over $F$.
Fix an algebraic closure of $F$ and write $G_F$ for the absolute Galois group $\Gal(\overline{F}/F)$.
For a given integer $m$, set $E[m]$ to be the Galois module of all $m$-torsion points in $E(\overline{F})$.
If $v$ is a prime in $F$, we write $F_v$ for the completion of $F$ at $v$.
The main object of interest is the Selmer group.
\begin{Defi}
For any integer $m\geq 2$, the \emph{$m$-Selmer group} is defined as follows
\[
\Sel_{m}(E/F) = \ker\left( H^1\left( G_{F}, E[m]\right) \rightarrow \prod_{v}H^1\left(G_{F_v}, E \right)[m]\right).
\]
\end{Defi}
\noindent This $m$-Selmer group fits into the following short exact sequence
\begin{equation}
\label{ses: m selmer}
0 \rightarrow E(F)/m E(F)\rightarrow \Sel_m(E/F) \rightarrow \Sha(E/F)[m] \rightarrow 0.
\end{equation}
Here, $\Sha(E/F)$ is the \emph{Shafarevich--Tate group} which is conjecturally finite.
Throughout this article, we will assume the finiteness of the Shafarevich--Tate group.

\subsection{Recollections from Iwasawa theory}
For details, we refer the reader to standard texts in Iwasawa theory (e.g. \cite[Chapter 13]{Was97}). 
Let $p$ be a fixed prime.
Consider the (unique) \emph{cyclotomic} $\Z_p$-extension of $\Q$, denoted by $\Q_{\cyc}$.
Set $\Gamma:=\Gal(\Q_{\cyc}/\Q)\simeq \Z_p$.
The \emph{Iwasawa algebra} $\Lambda$ is the completed group algebra $\Z_p\llbracket \Gamma \rrbracket :=\varprojlim_n \Z_p[\Gamma/\Gamma^{p^n}]$.
Fix a topological generator $\gamma$ of $\Gamma$; there is the following isomorphism of rings 
\begin{align*}
\Lambda&\xrightarrow{\sim}\Z_p\llbracket T\rrbracket \\
\gamma &\mapsto 1 +T.
\end{align*}

Let $M$ be a cofinitely generated cotorsion $\Lambda$-module.
The \emph{Structure Theorem of $\Lambda$-modules} asserts that the Pontryagin dual of $M$, denoted by $M^{\vee}$, is pseudo-isomorphic to a finite direct sum of cyclic $\Lambda$-modules.
In other words, there is a map of $\Lambda$-modules
\[
M^{\vee}\longrightarrow  \left(\bigoplus_{i=1}^s \Lambda/(p^{m_i})\right)\oplus \left(\bigoplus_{j=1}^t \Lambda/(h_j(T)) \right)
\]
with finite kernel and cokernel.
Here, $m_i>0$ and $h_j(T)$ is a distinguished polynomial (i.e., a monic polynomial with non-leading coefficients divisible by $p$).
The \emph{characteristic ideal} of $M^\vee$ is (up to a unit) generated by the \emph{characteristic element},
\[
f_{M}^{(p)}(T) := p^{\sum_{i} m_i} \prod_j h_j(T).
\]
The $\mu$-invariant of $M$ is defined as the power of $p$ in $f_{M}^{(p)}(T)$.
More precisely,
\[
\mu(M) = \mu_p(M):=\begin{cases}0 & \textrm{ if } s=0\\
\sum_{i=1}^s m_i & \textrm{ if } s>0.
\end{cases}
\]
The $\lambda$-invariant of $M$ is the degree of the characteristic element, i.e.
\[
\lambda(M) = \lambda_p(M) := \sum_{j=1}^t \deg h_j(T).
\]

Let $E_{/F}$ be an elliptic curve with good reduction at $p$.
We shall assume throughout that the prime $p$ is odd.
Let $N=N_E$ denote the conductor of $E$ and denote by $S$ the (finite) set of primes which divide $Np$.
Let $F_S$ be the maximal algebraic extension of $F$ which is unramified at the primes $v\notin S$.
Set $E[p^\infty]$ to be the Galois module of all $p$-power torsion points in $E(\overline{F})$.
For a prime $v\in S$ and any finite extension $L/F$ contained in the unique cyclotomic $\Z_p$-extension of $F$ (denoted by $F_{\cyc}$), write
\[
J_v(E/L) = \bigoplus_{w|v} H^1\left( G_{L_w}, E\right)[p^\infty]
\]
where the direct sum is over all primes $w$ of $L$ lying above $v$.
Then, the \emph{$p$-primary Selmer group over $L$} is defined as follows
\[
\Sel_{p^\infty}(E/L):=\ker\left\{ H^1\left(\Gal\left(F_S/L\right),E[p^{\infty}]\right)\longrightarrow \bigoplus_{v\in S} J_v(E/L)\right\}.
\]
It is easy to see that $\Sel_{p^\infty}(E/L) = \varinjlim_{n}\Sel_{p^n}(E/L)$, see for example \cite[\S 1.7]{CS00book}.
By taking direct limits of \eqref{ses: m selmer}, the $p$-primary Selmer group over $F$ fits into a short exact sequence
\begin{equation}
\label{sesSelmer}
0\rightarrow E(F)\otimes \Q_p/\Z_p\rightarrow \Sel_{p^{\infty}}(E/F)\rightarrow \Sha(E/F)[p^{\infty}]\rightarrow 0.
\end{equation}
Next, define
\[
J_v(E/F_{\cyc}) = \varinjlim J_v(E/L)
\]
where $L$ ranges over finite extensions contained in $F_{\cyc}$ and the inductive limit is taken with respect to the restriction maps.
The \emph{$p$-primary Selmer group over $F_{\cyc}$} is defined as follows

\[
\Sel_{p^\infty}(E/F_{\cyc}):=\ker\left\{ H^1\left(\Gal\left(F_S/F_{\cyc}\right),E[p^{\infty}]\right)\longrightarrow \bigoplus_{v\in S} J_v(E/F_{\cyc})\right\}.
\]
When $E$ is an elliptic curve defined over $\Q$, $p$ is an odd prime of good \emph{ordinary} reduction, and $F/\Q$ is an abelian extension, K.~Kato proved (see \cite[Theorem 14.4]{Kat04}) that the $p$-primary Selmer group $\Sel_{p^\infty}(E/F_{\cyc})$ is a cofinitely generated cotorsion $\Lambda$-module.
Therefore, in view of the Structure Theorem of $\Lambda$-modules, we can define the $\mu$ and $\lambda$-invariants, which we denote as $\mu_p(E/F)$ and $\lambda_p(E/F)$, respectively.

Given a number field $F$, we set $F_{\cyc}$ to be the composite of $F$ with $\Q_{\cyc}$.
It is the unique $\Z_p$-extension of $F$ which is contained in the infinite cyclotomic field $F(\mu_{p^\infty})$.
Given $n\geq 0$, let $F_n$ be the subfield of $F_{\cyc}$ such that $[F_n:F]=p^n$.
The following lemma relating the $\lambda$-invariant of the Selmer group to the rank of the elliptic curve is well-known but we include it for the sake of completeness. 
\begin{Lemma}
\label{lambda geq rank}
Let $E_{/F}$ be an elliptic curve and assume that $\Sel_{p^\infty}(E/F_{\cyc})$ is cotorsion as a $\Lambda$-module, and let $n\geq 0$.
Then, $\lambda_p(E/F)\geq \rank_{\Z} (E(F_n))$.
\end{Lemma}
\begin{proof}
Denote by $\Gamma_n$ the Galois group $\Gal(F_{\cyc}/F_n)$ and let $r_n$ denote the $\Z_p$-corank of $\Sel_{p^\infty}(E/F_{\cyc})^{\Gamma_n}$.
Since $\Sel_{p^\infty}(E/F_{\cyc})$ is cotorsion over the Iwasawa algebra, it has finite $\Z_p$-corank, and it is an easy consequence of the structure theorem that
\[
\lambda_p(E/F)=\corank_{\Z_p}\left( \Sel_{p^\infty}(E/F_{\cyc}) \right).
\]
We deduce from \eqref{sesSelmer} that 
\begin{equation}
\label{eq21}
\corank_{\Z_p} \Sel_{p^{\infty}}(E/F_n)\geq \rank_{\Z} (E(F_n)),
\end{equation}
with equality if $\Sha(E/F_n)[p^{\infty}]$ is finite.
It follows from the structure theory of $\Lambda$-modules that $\lambda_p(E/F_n)\geq r_n$.
It suffices to show that $r_n\geq \rank_{\Z} (E(F_n))$.
This is indeed the case, since Mazur's Control Theorem asserts that there is a natural map 
\[
\Sel_{p^{\infty}}(E/F_n)\rightarrow \Sel_{p^\infty}(E/F_{\cyc})^{\Gamma_n}
\]
with finite kernel.
From \eqref{eq21}, we see that $r_n\geq \rank_{\Z} (E(F_n))$ and the result follows.
\end{proof}

Let $L/\Q$ be a degree-$p$ Galois extension.
The following theorem relates  the $\lambda$-invariants of $\Selp$ and $\Sel_{p^\infty}(E/L_{\cyc})$.
\begin{Th}
\label{thm: Kida formula}
Let $p\geq 5$ be a fixed prime.
Let $L/\Q$ be a Galois extension of degree a power of $p$ disjoint from the cyclotomic $\Z_p$-extension of $\Q$.
Let $E_{/\Q}$ be a fixed elliptic curve with good ordinary reduction at $p$ and suppose that $\Selp$ is a cofinitely generated $\Z_p$-module.
Then, $\Sel_{p^\infty}(E/L_{\cyc})$ is also a cofinitely generated $\Z_p$-module.
Moreover, their respective $\lambda$-invariants are related by the following formula
\[
\lambda_p(E/L) = p\lambda_p(E/\Q) + \sum_{w \in P_1} \left(e_{L_{\cyc}/\Q_{\cyc}}(w) -1 \right) + 2\sum_{w\in P_2} \left(e_{L_{\cyc}/\Q_{\cyc}}(w) -1 \right)
\]
where $e_{L_{\cyc}/\Q_{\cyc}}(w)$ is the ramification index, and $P_1, P_2$ are sets of primes in $L_{\cyc}$ such that
\begin{align*}
P_1 & = \left\{ w : \ w\nmid p, \ E \textrm{ has split multiplicative reduction at }w \right\},\\
P_2 & = \left\{ w : \ w\nmid p, \ E \textrm{ has good reduction at }w, \ E(L_{\cyc,w}) \textrm{ has a point of order }p \right\}.
\end{align*}
\end{Th}

\begin{proof}
{\cite[Theorem 3.1]{HM99}}.
\end{proof}

\begin{Rem*}
We remind the reader that a cofinitely generated cotorsion $\Lambda$-module $M$ is a cofinitely generated $\Z_p$-module precisely when the associated $\mu$-invariant is 0.
\end{Rem*}

We record the following result which was first proven by Greenberg.
\begin{Prop}
\label{prop: greenberg}
Let $E_{/\Q}$ be a rank 0 elliptic curve and assume that the Shafarevich--Tate group is finite.
Then, for density one good (ordinary) primes, $\mu_p(E/\Q) = \lambda_p(E/\Q)= 0$.
\end{Prop}

\begin{proof}
See \cite[Theorem 5.1]{Gre99} or \cite[Theorem 3.7]{KR21}.
\end{proof}

\begin{Rem*}
In \cite[Corollary 3.6]{KR21}, it is shown that $\mu_p(E/\Q) = \lambda_p(E/\Q)= 0$ is equivalent to the vanishing of $\Sel_{p^\infty}(E/\Q_{\cyc})$.
This happens when $\Sel_{p^\infty}(E/\Q)=0$, the Tamagawa numbers at the primes of bad reduction of $E$ are not divisible by $p$, and $p$ is not an anomalous prime in the sense of \cite{Maz72}. 
\end{Rem*}

\subsection{Results, conjectures, and heuristics on ranks of elliptic curves}
\label{subsection: conjectures}
There are several important conjectures in the theory of elliptic curves.
The first one of interest is the \emph{rank distribution conjecture} which claims that over any number field, half of all elliptic curves (when ordered by height) have Mordell--Weil rank zero and the remaining half have Mordell--Weil rank one.
Finally, higher Mordell--Weil ranks constitute zero percent of all elliptic curves, even though there may exist infinitely many such elliptic curves.
Therefore, a suitably-defined \emph{average rank} would be $1/2$.
The best results in this direction are by M.~Bhargava and A.~Shankar (see \cite{BS15_quartic, BS15_cubic}).
They show that the average rank of elliptic curves over $\Q$ is strictly less than one, and that both rank zero and rank one cases comprise non-zero densities across all elliptic curves over $\Q$ (see \cite{BS13}).

Given an elliptic curve $E$ defined over $\Q$ and base-changed to $F$, we have the associated Hasse--Weil $L$ function, $L_E(s,F)$.
Let $F/\Q$ be an abelian extension with Galois group $G$ and conductor $f$.
Let $\hat{G}$ be the group of Dirichlet characters, $\chi: (\Z/f\Z)^\times \rightarrow \C^\times$.
We further know that
\[
L_E(s, F) = \prod_{\chi\in \hat{G}}L_E(s,\chi),
\]
where the terms appearing on the right hand side are the $L$-functions of $E_{/\Q}$ twisted by the character $\chi$.
\begin{Conj*}[Birch and Swinnerton-Dyer]
The Hasse--Weil $L$-function has analytic continuation to the whole complex plane, and
\[
\ord_{s=1} L_E(s,F) = \rank_{\Z}(E(F)).
\]
\end{Conj*}
It follows from the BSD Conjecture that the vanishing of the twisted $L$-functions $L_E(s,\chi)$ at $s=1$ is equivalent to the existence of rational points of infinite order on $E(F)$.

The following conjecture of David--Fearnley--Kisilevsky predicts that given an elliptic curve over $\Q$, the rank ``rarely'' jumps in $\Z/p\Z$-extensions with $p\neq 2$.
More precisely,
\begin{Conj*}[{\cite[Conjecture 1.2]{DFK07}}]
Let $p$ be an odd prime and $E_{/\Q}$ be an elliptic curve.
Define
\[
N_{E,p}(X) := \# \{\chi \textrm{ of order }p \mid \cond(\chi)\leq X \textrm{ and } L_{E}(1,\chi)=0 \}.
\]
\begin{enumerate}[\textup{(}1\textup{)}]
\item If $p=3$, then as $X\rightarrow\infty$
\[
\log N_{E,p}(X) \sim \frac{1}{2}\log X.
\]
\item If $p=5$, then as $X\rightarrow\infty$, the set $N_{E,p}(X)$ is unbounded but $N_{E,p}(X)\ll X^\epsilon$ for any $\epsilon>0$.
\item If $p\geq 7$, then $N_{E,p}(X)$ is bounded.
\end{enumerate}
\end{Conj*}
\noindent Under BSD, $\# N_{E,p}(X)$ can be rewritten as
\[
(p-1)\#\{ F/\Q \textrm{ is cyclic of degree }p \mid \cond(F)\leq X \textrm{ and } \rank_{\Z}(E(F))>\rank_{\Z}(E(\Q))  \}.
\]
In \cite[Theorem 1]{Dok07}, T.~Dokchitser showed that given an elliptic curve $E_{/\Q}$, there are infinitely many $\Z/3\Z$ extensions where the rank jumps.
More recently, B.~Mazur and K.~Rubin have shown (see \cite[Theorem 1.2]{MR18}) that given an elliptic curve $E$, there is a positive density set of primes (call it $\mathcal{S}$) such that for each $p\in \mathcal{S}$ there are infinitely many cyclic degree-$p$ extensions over $\Q$ with $\rank_{\Z}(E(L)) = \rank_{\Z}(E(\Q))$.
However, the result is unable to provide a positive proportion.
Mazur--Rubin revisited this conjecture in a recent preprint (see \cite{MR19_arithmetic_conjectures}) and their heuristics, based on the distribution of modular symbols, predicts the same statement as the conjecture of David--Fearnley--Kisilevsky.

\section{\texorpdfstring{Rank Jump in Degree-$p$ Galois Extensions}{}}
\label{section: no rank jump}
Let $E_{/\Q}$ be a rank 0 elliptic curve and $p$ be an odd prime number.
In this section, we are interested in studying two questions for a given pair $(E,p)$.
First, we analyze in how many (or what proportion of) cyclic degree-$p$ Galois extensions over $\Q$ does the rank of $E$ \emph{not} jump.
This question is addressed in Theorem~\ref{infinitely many p extensions with no lambda jump}: for non-CM elliptic curves, the result is unconditional; whereas, for the CM-case we prove the same result under an additional independence hypothesis \ref{assmpn: enemy primes for CM}.
Next, we study the \emph{dual problem}, i.e., for what proportion of elliptic curves does the rank \emph{not} jump in at least one degree-$p$ Galois extension over $\Q$.
This question is discussed in Section~\ref{section: dual problem}.

Even though such questions have been studied in the past, our approach involving Iwasawa theory is new.
Let $E_{/\Q}$ be an elliptic curve with \emph{good ordinary reduction} at a fixed odd prime $p$.
In Lemma \ref{lambda geq rank} we showed that over any number field, $\lambda_p(E/F)\geq \rank_{\Z}(E(F))$.
It is well-known (see for example, \cite[Corollary 3.6]{KR21}) that if $E_{/\Q}$ is a rank 0 elliptic curve with good (ordinary) reduction at an odd prime $p$ then $\mu_p(E/\Q) = \lambda_p(E/\Q) = 0$ if and only if $\Selp =0$.
We remind the reader that in Proposition \ref{prop: greenberg} we showed that the triviality of $\Selp$ is observed \emph{often}.
The same statement holds for any number field, under the additional hypothesis that the Shafarevich--Tate group is finite in every layer of its cyclotomic $\Z_p$-extension.
Our key idea is to start with a rank 0 elliptic curve $E_{/\Q}$ for which $\mu_p(E/\Q) = \lambda_p(E/\Q)=0$ and count how often is $\mu_p(E/L) = \lambda_p(E/L) =0$, where $L/\Q$ is a cyclic degree-$p$ extension.

For a number field $F$, it is possible that $\lambda_p(E/F)> \rank_{\Z}(E(F))$.
So, our method fails to measure \emph{all} instances when the rank of the elliptic curve does not change; i.e., our results only provide a lower bound.
However, we succeed in answering a stronger question, i.e., how often is the $p$-primary Selmer group $\Sel_{p^\infty}(E/F_{\cyc})$ trivial upon base-change.

\subsection{}
Let $E_{/\Q}$ be a fixed rank 0 elliptic curve of conductor $N$ with good ordinary reduction at a fixed prime $p\geq 5$ such that $\mu_p(E/\Q) = \lambda_p(E/\Q)=0$.
Recall that the \emph{rank distribution conjecture} predicts that half of the elliptic curves (ordered by height or conductor) have rank 0.
Moreover, Proposition \ref{prop: greenberg} asserts that there are density one good ordinary primes satisfying the condition of vanishing Iwasawa invariants.
Lastly, when an elliptic curve does \emph{not} have complex multiplication (CM), density one of the primes are good ordinary; in the CM case, the good ordinary and the good supersingular primes each have density 1/2.

Throughout this section, $L/\Q$ denotes a $\Z/p\Z$-extension disjoint from the cyclotomic $\Z_p$-extension.
Let $q$ be a prime number distinct from $p$ such that $E$ has good reduction at $q$, i.e., $\gcd(q,N)=1$.
Let $w|q$ be a prime in $L$, $L_w$ be the completion at $w$, and $\kappa$ be the residue field of characteristic $q$.
We know that there is an exact sequence of abelian groups (see \cite[Prop VII.2.1]{Sil09})
\[
0 \rightarrow E_1(L_w) \rightarrow E_0(L_w) \rightarrow  \widetilde{E}_{\ns}(\kappa) \rightarrow 0,
\]
where $\widetilde{E}_{\ns}(\kappa)$ is the set of non-singular points of the reduced elliptic curve, $E_0(L_w)$ is the set of points with non-singular reduction, and $E_1(L_w)$ is the kernel of reduction map.
Since $p\neq q$, we know that $E_1(L_w)[p]$ is trivial (see \cite[VII.3.1]{Sil09}).
Because $q$ is assumed to be a prime of good reduction $E_0(L_w) = E(L_w)$.
Hence,
\[
E(L_w)[p] \simeq \widetilde{E}(\kappa)[p].
\]
Since $L/\Q$ is a $\Z/p\Z$-extension, the residue field is either $\mathbb{F}_{q^p}$ or $\mathbb{F}_{q}$ depending on whether  $q$ is inert in the extension or not.

As explained above, (under standard hypothesis) we know that $\lambda_p(E/L)\geq \rank_{\Z}(E(L))$.
Since we have assumed that $\mu_p(E/\Q)=0$, it follows from Theorem \ref{thm: Kida formula} that $\mu_p(E/L)=0$.
To show that $\lambda_p(E/L)=0$, it suffices to show that
\[
\sum_{w\in P_1} \left( e_{L_{\cyc}/\Q_{\cyc}}(w) -1\right) = \sum_{w\in P_2} \left( e_{L_{\cyc}/\Q_{\cyc}}(w) -1\right)=0.
\]
Recall that all primes in the cyclotomic $\Z_p$-extension are finitely decomposed and the only primes that ramify are those above $p$.
Moreover, $L\cap \Q_{\cyc} =\Q$ and $p\not\in P_1 \cup P_2$ (recall the definition of these sets from Theorem \ref{thm: Kida formula}).
Therefore, $e_{L_{\cyc}/\Q_{\cyc}} = e_{L/\Q}$.
Since $p\geq 5$, the reduction type does not change upon base-change.
In particular, if $q(\neq p)$ is a prime of additive reduction for $E_{/\Q}$, then it has additive reduction over $L_{\cyc}$ (see \cite[p. 498]{ST68} or \cite[p. 587]{HM99}).
Finally, since $L_{\cyc,w}/L_w$ is a pro-$p$ group 
we have that $E(L_{\cyc,w})[p^\infty]=0$ if and only if $E(L_{w})[p^\infty]=0$.
Thus, it \emph{suffices} to show that
\begin{equation}
\label{eqn: last term is trivial}
\sum_{w\in P_1} \left( e_{L/\Q}(w) -1\right) = \sum_{w\in P_2} \left( e_{L/\Q}(w) -1\right)=0,
\end{equation} 
where $P_1, P_2$ are now sets of primes in $L$.
More precisely,
\begin{align*}
P_1 & = \left\{ w\in L : \ w\nmid p, \ E \textrm{ has split multiplicative reduction at }w \right\},\\
P_2 & = \left\{ w\in L : \ w\nmid p, \ E \textrm{ has good reduction at }w, \ E(L_{w}) \textrm{ has a point of order }p \right\}.
\end{align*}
It is often possible that $P_1 = \emptyset$ but the set $P_2$ is never empty.
In fact, it is known that (see for example \cite[\S 2]{Coj04})
\begin{equation}
\label{eqn from cojocaru}
\lim_{X\rightarrow\infty} \frac{\# \{q \leq X \mid q\nmid pN, \ p\mid \#\widetilde{E}(\F_q)\}}{\pi(X)} \approx \frac{1}{p}.
\end{equation}
Here, $\pi(X)$ is the prime counting function.

\begin{Rem*}
Henceforth, the sets $P_1, P_2$ will denote sets of primes in $L$ (rather than $L_{\cyc}$).
By our assumption that $L\cap \Q_{\cyc}=\Q$, we have excluded the case that $p$ is the only ramified prime.
If two or more primes are ramified, then it is possible that $p$ is (wildly) ramified in $L$.
But by definition, $p\not\in P_1\cup P_2$.
Therefore, the ramification of $p$ in $L/\Q$ does not contribute to the $\lambda$-jump.
It suffices to focus on the ramification of primes $q\neq p$.
\end{Rem*}
The above discussion can be summarized in the result below.
\begin{Prop}
\label{prop: no rank jump}
Let $E_{/\Q}$ be a rank 0 elliptic curve with good ordinary reduction at $p\geq 5$ such that $\mu_p(E/\Q)=\lambda_p(E/\Q)=0$.
Let $L/\Q$ be any cyclic degree-$p$ extension disjoint from $\Q_{\cyc}$ such that \eqref{eqn: last term is trivial} holds.
Then $\mu_p(E/L) = \lambda_p(E/L)=0$.
This implies in particular that $\rank_{\Z}( E(L_n))=0$ for all $n\geq 0$.
\end{Prop}

The conditions imposed in Proposition \ref{prop: no rank jump} will be required throughout this section.
Therefore, we make the following definition.
\begin{Defi}
\label{defn: relevant and irrelevant}
Given an elliptic curve $E_{/\Q}$ of rank 0, a prime $p$ is called \emph{irrelevant} if at least one of the following properties hold.
\begin{enumerate}[\textup{(}i\textup{)}]
\item $p$ is a prime of bad reduction.
\item $p$ is a prime of supersingular reduction.
\item At $p$, the $\mu$-invariant associated to the $p$-primary Selmer group is positive.
\item At $p$, the $\lambda$-invariant associated to the $p$-primary Selmer group is positive.
\end{enumerate}
Otherwise, it is called a \emph{relevant prime}.
\end{Defi}

Let $E$ be an elliptic curve and let $p,q$ be two distinct primes.
Given a triple $(E, p, q)$ we aim to understand when \eqref{eqn: last term is trivial} holds.
We begin by recalling the following well-known result.

\begin{Prop}
\label{well know prop from jr08}
Let $p$ be an odd prime.
Let $L/\Q$ be any $\Z/p\Z$-extension disjoint from the cyclotomic $\Z_p$-extension which is ramified at exactly one prime $q\neq p$.
Such an extension exists precisely when $q \equiv 1 \pmod{p}$.
Moreover, $L$ is unique and has conductor $q$.
\end{Prop}
\begin{proof}
See for example \cite[Proposition 1.1]{JR08}.
\end{proof}

In fact, it follows from class field theory (see \cite[Lemma 2.5]{MM16}) that the primes that ramify in a $\Z/p\Z$-extension are either $p$ or precisely those of the form $q\equiv 1\pmod{p}$.
By local class field theory, the discriminant of $L/\Q$ (denoted $d(L/\Q)$) is given by (see \cite[Lemma 2.4]{MM16})
\[
d(L/\Q) = \begin{cases}
\prod_{i=1}^r q_i^{p-1} & \textrm{ if }q_i's \textrm{ are ramified.}\\
\prod_{i=1}^r q_i^{p-1}p^{2(p-1)} & \textrm{ if }q_i's \textrm{ and } p \textrm{ are ramified.}
\end{cases}
\]
If $q(\neq p)$ ramifies in a $\Z/p$-extension then the ramification is tame.

For primes of the form $q\equiv 1\pmod{p}$ (i.e., primes that can ramify in $\Z/p\Z$-extensions), we introduce the notion of \emph{friendly} and \emph{enemy primes}.
\begin{Defi}
Let $p$ be a fixed odd prime and $E_{/\Q}$ be a fixed elliptic curve of rank 0 for which $p$ is a \emph{relevant prime} (see Definition \ref{defn: relevant and irrelevant}).
Define \emph{enemy primes} to be those primes which are of the form $q\equiv 1\pmod{p}$ and such that either of the following conditions hold
\begin{enumerate}[\textup{(}i\textup{)}]
\item $q$ is a prime of split multiplicative reduction \emph{or}
\item $q$ is a prime of good reduction and $p \mid \# \widetilde{E}(\mathbb{F}_q)$.
\end{enumerate}
If a prime is of the form $q\equiv 1 \pmod{p}$ with the additional properties that $q$ is a prime of good reduction and $p\nmid \#\widetilde{E}(\mathbb{F}_q)$, then $q$ will be called \emph{friendly}.
\end{Defi}

\begin{Rem*}
\begin{enumerate}[\textup{(}i\textup{)}]
\item A prime $q\equiv 1\pmod p$ which is a prime of bad reduction but not of split multiplicative type is neither an \emph{enemy prime} nor a \emph{friendly prime}.
\item For our purpose, it is enough to work with the residue field $\mathbb{F}_q$ because eventually we want the primes $q$ to \emph{ramify} in the extension $L$.
\end{enumerate}
\end{Rem*}

Primes $w|q$ in $L$ will also be called an \emph{enemy} or a \emph{friendly} prime depending on the behaviour of $q$.
The following lemma will play a crucial role in the subsequent discussion.

\begin{Lemma}
\label{enemy implies jump}
Let $L/\Q$ be a cyclic degree-$p$ extension disjoint from $\Q_{\cyc}$.
Then \eqref{eqn: last term is trivial} holds precisely when no ramified prime is an \emph{enemy prime}.

\end{Lemma}
\begin{proof}
Recall that $P_1$ consists of primes $w\nmid p$ such that $E$ \emph{has} split multiplicative reduction at $w$.
Observe that
\[
\sum_{w\in P_1}\left( e_{L/\Q}(w) -1 \right) \neq 0
\]
if and only if there exists a ramified prime in $L/\Q$ of split multiplicative type.
The assertion is immediate from the definition of an \emph{enemy prime}.

Since $L$ is disjoint from $\Q_{\cyc}$, we know that if $p$ is ramified in $L/\Q$, there must be at least one other prime which is also ramified.
Now, 
\[
\sum_{w\in P_2}\left( e_{L/\Q}(w) -1 \right)\neq 0
\]
precisely when there exists a $q(\neq p)$ satisfying \emph{all} of the following conditions:
\begin{enumerate}[\textup{(}i\textup{)}]
\item $q$ is a prime of good reduction for $E$,
\item $q$ is ramified in the extension $L/\Q$, \emph{and}
\item $w$ is a prime above $q$ with $E(L_w)[p]\neq 0$,
\end{enumerate} 
The conclusion of the lemma is now straightforward.
\end{proof}

\subsubsection{} For a given pair $(E,p)$, we denote the set of \emph{enemy primes} by $\cE_{(E,p)}$ and the set of \emph{friendly primes} by $\cF_{(E,p)}$.
Write $\mathcal{N}_{(E,p)}$ for the set of primes of the form $q\equiv 1\pmod{p}$ where $E$ has bad reduction not of split multiplicative type.
The three sets are disjoint.
We further subdivide the first two of these sets into disjoint sets, namely
\begin{align*}
\cF_{(E,p)} &= \cF_{(E,p)}^{\ord} \cup \cF_{(E,p)}^{\super} \textrm{ and }\\
\cE_{(E,p)} &= \cE_{(E,p)}^{\spl} \cup \cE_{(E,p)}^{\ord} \cup \cE_{(E,p)}^{\super}.
\end{align*}
Here, $\cF_{(E,p)}^{\ord}$ (resp. $\cF_{(E,p)}^{\super}$) is the set of primes of the form $q\equiv 1 \pmod{p}$ such that $q$ is a prime of good \emph{ordinary} (resp. \emph{supersingular}) reduction and $p\nmid \#\widetilde{E}(\F_q)$.
The set $\cE_{(E,p)}^{\spl}$ consists of all the primes $q \equiv 1\mod{p}$ of split multiplicative reduction.
Finally, $\cE_{(E,p)}^{\ord}$ (resp. $\cE_{(E,p)}^{\super}$) is the set of primes of the form $q\equiv 1 \pmod{p}$ such that $q$ is a prime of good \emph{ordinary} (resp. \emph{supersingular}) reduction and $p| \#\widetilde{E}(\F_q)$.

\begin{Lemma}\label{no supersingular primes lemma}
Let $E_{/\Q}$ be an elliptic curve and $p\geq 5$ be a \emph{relevant prime}.
Then, $\cE_{(E,p)}^{\super} = \emptyset$.
In particular, $\cE_{(E,p)} = \cE_{(E,p)}^{\spl} \cup \cE_{(E,p)}^{\ord}$.
\end{Lemma}

\begin{proof}
When $q \geq 7$ is a prime of supersingular reduction then it follows from the Hasse bound that $a_q =0$.
Therefore, 
\[
\# \widetilde{E}(\F_q) =  q+ 1 - a_ q = q+1.
\]
Note that we require that $q \equiv 1\pmod{p}$.
Thus, $p \nmid \# \widetilde{E}(\F_q)$. 
\end{proof}

For any subset $\mathcal{S}'$ of the set of primes, let $\mathfrak{d}(\mathcal{S}')$ denote the Dirichlet density of $\mathcal{S}'$.
With notation as above, we have that
\begin{equation}
\label{eqn: add enemy and friendly}
\mathfrak{d}\left(\cE_{(E,p)}\right) + \mathfrak{d}\left(\cF_{(E,p)}\right) + \mathfrak{d}\left(\mathcal{N}_{(E,p)}\right)= \frac{1}{\varphi(p)} = \frac{1}{p-1}.
\end{equation}
The first equality follows from Dirichlet's theorem on primes in arithmetic progressions, which asserts that the proportion of primes that are congruent to $1$ modulo $p$ is $1/\varphi(p)$.
The set $\mathcal{N}_{(E,p)}$ is finite, hence $\mathfrak{d}\left(\mathcal{N}_{(E,p)}\right) = 0$.
We will henceforth disregard the primes $q$ which are of non-split multiplicative and additive reduction type.
For the same reason, we may also disregard (for the purpose of proportion) the primes of split multiplicative reduction.
As $X\rightarrow \infty$, the contribution to $\mathcal{E}_{(E,p)}$ is primarily from primes $q$ such that
\begin{enumerate}[\textup{(}a\textup{)}]
\item $q$ is a prime of good ordinary reduction, 
\item $q\equiv 1\pmod{p}$, \emph{and}
\item $p\mid\# \widetilde{E}(\F_q)$.
\end{enumerate}
For non-CM elliptic curves, the set of primes of good ordinary reduction has density $1$.
In the case of non-CM elliptic curves with surjective residual Galois representation at $p$, we know how to calculate the proportion of primes satisfying conditions (a)--(c).

\begin{Lemma}
Let $E_{/\Q}$ be a non-CM elliptic curve and $p\geq 5$ be a fixed prime of good ordinary reduction such that the residual Galois representation at $p$ is surjective.
Then,
\[
\mathfrak{d}\left(\cE_{(E,p)}\right)= 
\frac{p}{(p-1)^2(p+1)}.
\]
\end{Lemma}

\begin{proof}
The result is well known and follows from the proof of \cite[Proposition~4.6]{GFP20}. For the convenience of the reader, we shall briefly sketch the details here. A more detailed argument is provided in Section \ref{section 4.2}. Since $p\geq 5$, it follows from the proof of Lemma \ref{no supersingular primes lemma} that if conditions (b) and (c) are satisfied, then (a) is automatically satisfied for any prime $q$ of good reduction. Since there are only finitely many primes $q$ at which $E$ has bad reduction, we may as well assume that $q$ is a prime of good reduction.
\par Let $\Frob_q$ denote the Frobenius at $q$ and set $\mathcal{S}$ to be the set of matrices $A\in \GL_2(\F_p)$ such that $\trace(A)=2$ and $\det(A)=1$. Let $\bar{\rho}:\Gal(\bar{\Q}/\Q)\rightarrow \GL_2(\F_p)$ denote the residual representation at $p$, i.e., the representation on the group of $p$-torsion points $E[p]$. By assumption, $\bar{\rho}$ is surjective.
\par Since $q$ is a prime of good reduction, $\bar{\rho}$ is unramified at $q$ and the characteristic polynomial of $\bar{\rho}(\Frob_q)$ is given by
\[\det\left(\Id-T \bar{\rho}(\Frob_q)\right)=T^2-(q+1-\#\tilde{E}(\F_q)) T+q.\]
Thus, $q$ is satisfies both conditions (b) and (c) above if and only if $\bar{\rho}(\Frob_q)$ is contained in $\mathcal{S}$. According to Lemma \ref{S cardinality}, the cardinality of $\mathcal{S}$ is $p^2$. The cardinality of $\image(\bar{\rho})=\GL_2(\F_p)$ is $(p^2-1)(p^2-p)$. The result follows from the Chebotarev density theorem, according to which the density of $\mathcal{E}_{(E,p)}$ is \[\frac{\#\mathcal{S}}{\# \GL_2(\F_q)}=\frac{p}{(p-1)^2(p+1)}.\]
\end{proof}

We performed calculations using SageMath \cite{sagemath} to get an estimate of the proportion of \emph{enemy primes} in the CM case.
More precisely, fix $5 \leq p\leq 50$.
Fix an elliptic curve $E_{/\Q}$ of rank 0 and conductor less than 100.
Running through all primes $q$ less than 200 million, we computed the proportion of \emph{enemy primes}.
The results are recorded at the end of the article, in Table \ref{table: enemy primes}.
The data suggests that for elliptic curves with CM, the proportion of \emph{enemy primes} is \emph{half} of that in the non-CM case.
We know from Deuring's Criterion that for a CM elliptic curve, the density of the set of good ordinary primes is $1/2$.
It therefore, seems reasonable to assume that the \emph{enemy primes} are equally likely to be primes of good ordinary or good supersingular reduction.
More precisely,

\begin{equation}\tag*{\textup{\textbf{Hyp~CM}}}\label{assmpn: enemy primes for CM}
\begin{minipage}{0.85\textwidth}
Let $E_{/\Q}$ be an elliptic curve with CM and $p$ be a fixed odd prime of good ordinary reduction.
Then, $\mathfrak{d}\left(\cE_{(E,p)}\right)= 
\frac{p}{2(p-1)^2(p+1)}$.
\end{minipage}
\end{equation}

We can now prove the main result of this section.

\begin{Th}\label{rev Th 3.8}
\label{infinitely many p extensions with no lambda jump}
Let $(E,p)$ be a given pair of a rank 0 elliptic curve over $\Q$ and a \emph{relevant prime} $p\geq 5$.
Then the following assertions hold.
\begin{enumerate}[\textup{(}1\textup{)}]
\item Suppose that the residual representation at $p$ is surjective.
Then, there are infinitely many cyclic number fields of degree-$p$ in which the $\lambda$-invariant does not jump.
In particular, there are infinitely many $\Z/p\Z$-extensions $L/\Q$ in which the rank does not jump in $L_n$ for all $n\geq 0$.
\item Let $(E,p)$ be a pair of rank 0 elliptic curve (defined over $\Q$) with complex multiplication and $p\geq 5$ be a relevant prime.
Suppose further that \ref{assmpn: enemy primes for CM} holds.
Then the same conclusions hold as in the non-CM case.
\end{enumerate}
\end{Th}

\begin{proof}
Since $p$ is assumed to be a relevant prime, we know that $\mu_p(E/\Q)=\lambda_p(E/\Q)=0$.
It follows from Theorem~\ref{thm: Kida formula} that $\mu_p(E/L)=0$ for every degree $p$ extension $L/\Q$.
The result will follow from Proposition \ref{prop: no rank jump} if we can show that there are infinitely many cyclic degree $p$ extensions such that \eqref{eqn: last term is trivial} holds.

Observe that
\[
\mathfrak{d}\left(\cF_{(E,p)}\right) = \frac{1}{p-1} - \mathfrak{d}\left(\cE_{(E,p)}\right) = 
\begin{cases}
\frac{p^2-p-1}{(p-1)^2(p+1)} & \text{ if }E\text{ does not have CM.}\\
\frac{2p^2-p-2}{2(p-1)^2(p+1)} & \text{ if }E\text{ does not have CM.}
\end{cases}
\]
It follows from Lemma \ref{enemy implies jump} that
\eqref{eqn: last term is trivial} holds when every ramified prime is a \emph{friendly prime}.
From the above discussion, we see that there are infinitely many \emph{friendly primes}.
We have also shown in Proposition~\ref{well know prop from jr08} that corresponding to each \emph{friendly prime} (say $q$) there is one $\Z/p\Z$-extension where only $q$ ramifies.
This completes the proof.
\end{proof}

Note that the condition in Theorem \ref{rev Th 3.8}, part (1) requiring that the residual representation is surjective implies that the elliptic curve does not have complex multiplication. The Galois representations associated to CM elliptic curves are studied for instance in \cite{lozanocm}.
The following application of our theorem was pointed out by J.~Morrow.
We begin by stating a result of Gonz{\'a}lez-Jim{\'e}nez--Najman.
\begin{Th}
\label{result from GJN}
Let $E_{/\Q}$ be an elliptic curve and $p>7$ be a prime number.
Let $L/\Q$ be a Galois extension with Galois group $G\simeq \Z/p\Z$.
Then $E(L)_{\tors}=E(\Q)_{\tors}$.
\end{Th}

\begin{proof}
See \cite[Theorem 7.2]{GJN20}.
\end{proof}

Combining Theorems \ref{infinitely many p extensions with no lambda jump} and \ref{result from GJN}, the following corollary is immediate.

\begin{Cor}\label{rev cor dioph stability}
Let $(E,p)$ be a given pair of rank 0 elliptic curve over $\Q$ and a relevant prime $p> 7$.
If $E_{/\Q}$ is an elliptic curve without CM, suppose that the residual representation at $p$ is surjective.
If $E_{/\Q}$ is an elliptic curve with CM, suppose that \ref{assmpn: enemy primes for CM} holds.
Then, there are infinitely many $\Z/p\Z$-extensions of $\Q$ where the Mordell--Weil group does not grow.
\end{Cor}

\begin{Rem*}
To show that there are ``infinitely many'' $\Z/p\Z$-extensions with no $\lambda$-jump, we only counted those where \emph{exactly one friendly prime} is ramified.
In particular, we counted only those $\Z/p\Z$-extensions which are contained in the cyclotomic field $\Q(\mu_q)$ where $q$ is a prime of the form $1 \pmod{p}$.
Our count ignored the contribution from $\Z/p\Z$-extensions where two or more primes are ramified, all of which are either \emph{friendly primes} or the prime $p$.
It was pointed out to us by R.~Lemke Oliver that Theorem \ref{infinitely many p extensions with no lambda jump} can be made explicit using standard analytic number theory techniques (see for example \cite[Th{\'e}or{\`e}me~2.4]{Ser74}); 
however, our approach is still likely to fall short of proving a positive proportion. 

A recent and significant result in this direction is by Mazur--Rubin (see \cite[Theorem~1.2]{MR18}).
They show that given an elliptic curve $E$, there is a positive density set of primes (call it $\mathcal{S}$) such that for each $p\in \mathcal{S}$ there are infinitely many $\Z/p\Z$-extensions over $\Q$ with $E(L) = E(\Q)$.
In the case of rank 0 elliptic curves (defined over $\Q$), our result is stronger, in the sense that for $p>7$, Corollary \ref{rev cor dioph stability} holds unconditionally for density 1 primes if $E$ is an elliptic curve without CM.
On the other hand, for CM elliptic curves, the result is conditional and applies to a set of primes of density $1/2$.
\end{Rem*}

\subsubsection{Example: CM case}
\label{CM Example}
Now, we work out a particular example in the CM case.
As before, let $p\geq 5$ be a fixed prime and $q\equiv 1 \pmod{p}$ be a different prime.
Let $k\not\equiv 0 \pmod{q}$ and consider the family of curves
\[
E_k: y^2 = x^3 - kx.
\]
Then either of the two statements are true (see \cite[\S 4.4]{Was08}).
\begin{enumerate}[\textup{(}i\textup{)}]
\item If $q\equiv 3 \pmod{4}$, then $\# \widetilde{E_k}(\mathbb{F}_q) = q +1$.
\item If $q\equiv 1 \pmod{4}$, write $q= s^2 + t^2$ with $s\in \Z$, $t\in 2\Z$ and $s+t \equiv 1 \pmod{4}$.
Then 
\[
\# \widetilde{E_k}(\mathbb{F}_q) = \begin{cases}
q + 1 - 2s & \textrm{ if } k \textrm{ is a fourth power mod }q\\
q + 1 + 2s & \textrm{ if } k \textrm{ is a square mod }q \textrm{ but not a fourth power}\\  
q + 1 \pm 2t & \textrm{ if } k \textrm{ is not a square mod }q.
\end{cases}
\]
\end{enumerate}
For this family of elliptic curves, the primes $q\equiv 3 \pmod{4}$ are \emph{supersingular}.
Recall from Lemma \ref{no supersingular primes lemma} that if $q$ is a supersingular prime and $q\equiv 1 \pmod{p}$ then the primes $w|q$ are \emph{friendly} (except for possibly finitely many).
For the sake of concreteness, suppose that $k=1$ (the argument goes through more generally).
By the Chinese Remainder Theorem we know that if 
\begin{align*}
q&\equiv 3 \pmod{4} \textrm{ and }\\
q&\equiv 1 \pmod{p} 
\end{align*}
then, $q \equiv 2p+1 \pmod{4p}$.
By Dirichlet's theorem for primes in arithmetic progressions, we know that there are infinitely many primes satisfying this congruence condition.
Moreover, the proportion of such primes is 
\[
\frac{1}{\varphi(4p)} = \frac{1}{2(p-1)}.
\]
By Proposition \ref{well know prop from jr08}, we know that corresponding to each such $q$ there is \emph{exactly one} cyclic degree-$p$ extension $L/\Q$ in which $q$ is the unique ramified prime.
Therefore, we have produced infinitely many $\Z/p\Z$-extensions where the ramified primes are \emph{friendly}.
By Proposition \ref{prop: no rank jump}, if $\rank_{\Z}(E_k/\Q) =0$ and $\mu_p(E/\Q)=0$ (e.g. when $k=1$) then there are infinitely many $\Z/p\Z$-extensions over $\Q$ where the rank does \emph{not} jump.

We record a specific case of the above discussion.
\begin{Th}
Let $p\geq 5$ be a fixed prime, and consider the elliptic curve
\[
E: y^2 = x^3 - x.
\]
Then, there are infinitely many $\Z/p\Z$-extensions $L/\Q$ such that $\Sel_{p^\infty}(E/L_{\cyc})=0$.
In particular, there are infinitely many $\Z/p\Z$-extensions such that $\rank_{\Z}(E(L))=0$.
\end{Th}

\subsection{}
\label{section: dual problem}
In the last section, we fixed a rank 0 elliptic curve over $\Q$ and analyzed \emph{in how many} $\Z/p\Z$-extensions of $\Q$ did the rank jump.
Now, we ask the following question.
\begin{Ques*}
Let $p\geq 5$ be a fixed odd prime.
For what proportion of rank 0 elliptic curves does there exist \emph{at least one} degree-$p$ Galois extension over $\Q$ disjoint from $\Q_{\cyc}$ such that its rank remains 0 upon base-change? 
\end{Ques*}

Let $E$ be a rank 0 elliptic curve of conductor $N$ and $p$ be a \emph{relevant prime}.
Throughout this section we assume that the Shafarevich--Tate group of rank 0 elliptic curves defined over $\Q$ is finite.
Proposition \ref{prop: greenberg} asserts that density one good ordinary primes are \emph{relevant}.
Further assume that $N$ is divisible by at least one prime of the form $1\pmod{p}$.
The following lemma is a special case of Proposition \ref{prop: no rank jump}.

\begin{Lemma}
\label{lemma: basic}
Keep the setting as above.
If $E$ has no prime of split multiplicative reduction, i.e., $P_1 = \emptyset$, then there exists at least one degree-$p$ cyclic extension $L/\Q$ such that $\lambda_p(E/L)=0$.
In particular, $\rank_{\Z}(E(L))=0$.
\end{Lemma}
\begin{proof}
By Proposition \ref{well know prop from jr08}, there exists one and only one cyclic degree-$p$ extension $L/\Q$ of conductor $q$ if and only if $q\equiv 1 \pmod{p}$.
By assumption, $N$ has a prime divisor of this form (say $q_0$).
Let $L/\Q$ be the $\Z/p\Z$-extension of conductor $q_0$.
Moreover, there is no contribution from the last term of the formula in Theorem \ref{thm: Kida formula}; indeed, for any $w\in P_2$ the ramification index is $e_{L/\Q}(w)=1$. 
Since $L$ is a $\Z/p\Z$-extension, it follows from Theorem \ref{thm: Kida formula} that $\mu_p(E/L)=0$.
In this $\Z/p\Z$-extension, we have forced $\lambda_p(E/L)=0$.
Therefore, $\Sel_{p^\infty}(E/L_{\cyc})=0$.
Clearly, $\Sel_{p^\infty}(E/L)$ is trivial as well; this forces $\rank_{\Z}(E(L))=0$.
\end{proof}

It follows from the proof of the above result that to obtain a lower bound of the proportion of rank 0 elliptic curves for which there is \emph{at least one} degree-$p$ cyclic extension over $\Q$ such that $\Sel_{p^\infty}(E/L)=0$, it is enough to count elliptic curves with the following properties
\begin{enumerate}[\textup{(}i\textup{)}]
\item $E$ has good ordinary reduction at $p\geq 5$.
\item $E$ has \emph{any} reduction type at primes $q\not\equiv 1 \pmod{p}$.
\item $E$ has at least one prime of bad reduction which is not of split multiplicative type at a prime $q\equiv 1\pmod{p}$.
\end{enumerate}

Before proceeding with such a count, we need to define the notion of height.
For an elliptic curve defined over $\Q$, consider the long Weierstrass equation 
\[
y^2+a_1 xy+a_3 y=x^3+a_2 x^2+a_4 x+a_6.
\]
Define the \emph{height} of this Weierstrass equation with integer coefficients $\boldsymbol{a}=(a_1, a_2, a_3, a_4, a_6)$ to be
\[
\height(\boldsymbol{a}) = \max_i \left\{\abs{a_i}^{1/i}\right\}.
\]
Let $S$ be a set of Weierstrass equations with integer coefficients $\boldsymbol{a}$ which are ordered by height.
The \emph{proportion of Weierstrass equations which lie in the set $S$} is defined as
\begin{equation}
\label{eqn: height defi}
\lim_{X\rightarrow \infty}\frac{\#\left\{ \boldsymbol{a}\in S : \height(\boldsymbol{a}) <X\right\}}{\#\left\{ \boldsymbol{a}\in \Z^5 : \height(\boldsymbol{a}) <X\right\}}.
\end{equation}
For our purposes, we restrict to Weierstrass equations which are globally minimal.

\begin{Lemma}
\label{lemma: reduction type proportions}
Let $q$ be any prime.
Suppose that all elliptic curves defined over $\Q$ are ordered by height.
Then, 
\begin{enumerate}[\textup{(}i\textup{)}]
\item the proportion with split multiplicative reduction at $q$ is $
\frac{q-1}{2q^2}$.
\item the proportion with good reduction at $q$ is $\left(1 - \frac{1}{q}\right)$.
\end{enumerate} 
\end{Lemma}

\begin{proof}
For (i), see \cite[Theorem 5.1]{CS20}.
For (ii), see \cite[Proposition 2.2]{CS20}.
\end{proof}

Henceforth, we will assume that the rank of the elliptic curve and the reduction type at $\ell$ are independent of each other.
In other words, we will assume that even if we only order the rank 0 elliptic curves by height, the proportion of elliptic curves with split multiplicative reduction or good reduction is the same as that in Lemma~\ref{lemma: reduction type proportions}. 

We know from \cite[Section~3]{CS20} that the local conditions such as the reduction type of elliptic curves at distinct primes are independent.

We now record the assumptions we have made:
\begin{equation}\tag*{\textup{\textbf{Hyp~ind}}}\label{assmpn: for rank 0 indep}
\begin{minipage}{0.85\textwidth}
The reduction type of an elliptic curve at a prime $q$ is independent of its rank.
\end{minipage}
\end{equation}
What we mean is that density results from Lemma~\ref{lemma: reduction type proportions} hold even when we restrict to rank 0 elliptic curves defined over $\Q$.
We can now prove the main result in this section.

\begin{Th}
\label{thm: at least one degree p extension with trivial selmer}
Let $p\geq 5$ be a fixed odd prime.
Suppose that \ref{assmpn: for rank 0 indep} holds. 
Varying over all rank 0 elliptic curves defined over $\Q$ and ordered by height, the proportion with 
\begin{enumerate}[\textup{(}i\textup{)}]
\item good reduction at $p$,
\item any reduction type at $q\not\equiv 1 \pmod{p}$, \emph{and}
\item at least one prime of bad reduction where the reduction type is \emph{not} split multiplicative at a prime $q\equiv 1\pmod{p}$ is given by
\end{enumerate}
\begin{equation}
\label{eqn: expression}
\left(1 - \frac{1}{p} \right)
\left( 1 - \prod_{q \equiv 1 (\textrm{mod }{p})}
\left( \frac{q-1}{2q^2} + 1 - \frac{1}{q}\right) \right).
\end{equation}
Further, suppose that the Shafarevich--Tate group is finite for all rank 0 elliptic curves defined over $\Q$.
There is a positive proportion of rank 0 elliptic curves for which there is \emph{at least one} degree-$p$ cyclic extension over $\Q$ such that $\Sel_{p^\infty}(E/L)=0$ upon base-change.
\end{Th}

\begin{proof}
Let $q \equiv 1\pmod{p}$.
We require that among all such primes there is ``at least one prime of bad reduction where the reduction type is \emph{not} split multiplicative''.
Equivalently, among all primes of the form $1\pmod{p}$ there is ``at least one prime of additive or non-split multiplicative reduction''.
In other words, at $q\equiv 1 \pmod{p}$ we want the \emph{negation} of ``all primes have either good or split multiplicative reduction''.
This gives \eqref{eqn: expression} from Lemma \ref{lemma: reduction type proportions}.

From our earlier discussion, to prove the second assertion it remains to show that \eqref{eqn: expression} is strictly positive.
For any fixed prime $p$, note that
\[
\prod_{q \equiv 1 (\textrm{mod }{p})} \left( \frac{q-1}{2q^2} + 1 - \frac{1}{q}\right) 
= \prod_{q \equiv 1 (\textrm{mod }{p})} \left( 1 - \left(\frac{q+1}{2q^2}\right)\right)< 1.
\]
The inequality follows from the fact that each term in the product is $<1$.
Therefore,
\[
\left( 1 - \prod_{q \equiv 1 (\textrm{mod }{p})} \left( 1 - \left(\frac{q+1}{2q^2}\right)\right)\right) >0.
\]
The claim follows.
\end{proof}

In Table \ref{table: values for positive prop}, we compute the expression \eqref{eqn: expression} for $3\leq p < 50$ and $q \leq 179,424,673$ (i.e., the first 10 million primes).

A reasonable question to ask at this point is the following.
\begin{Ques*}
Let $p\geq 5$ be a fixed prime and $L/\Q$ be a fixed $\Z/p\Z$-extension.
Varying over all rank 0 elliptic curves over $\Q$ ordered by height, for what proportion is $\Sel_{p^\infty}(E/L_{\cyc})=0$?
\end{Ques*}

In this direction, we can provide partial answers.
Given a number field $L/\Q$, we can find a lower bound for the proportion of elliptic curves (defined over $\Q$) for which \eqref{eqn: last term is trivial} holds.
\begin{Prop}
Let $p$ be a fixed odd prime.
Let $L/\Q$ be a fixed $\Z/p\Z$-extension which is tamely ramified at primes $q_1, \ldots, q_r$.
The proportion of elliptic curves defined over $\Q$ (ordered by height) such that \eqref{eqn: last term is trivial} holds has a lower bound of
\[
\prod_{q_i=1}^r \left( \frac{q_i+1}{2q_i^2}\right).
\]
\end{Prop}

\begin{proof}
Note that for \eqref{eqn: last term is trivial} to hold, the reduction type at $p$ does not matter.
To find a lower bound, it suffices that \emph{all} the ramified primes (i.e., the $q_i$'s) are primes of bad reduction of non-split multiplicative or additive reduction type.
Indeed, this would ensure $P_1 =\emptyset$ and the primes of good reduction are not ramified.
Equivalently, each $q_i$ is such that it does \emph{not} have good ordinary or split multiplicative reduction.
By Lemma \ref{lemma: reduction type proportions}, the lower bound is
\[
\prod_{q_i=1}^r \left( 1 - \left(\frac{q_i -1}{2q_i^2} + 1 - \frac{1}{q_i}\right)\right) = \prod_{q_i=1}^r \left( \frac{q_i+1}{2q_i^2}\right).
\]
\end{proof}

\begin{Rem*}
The above proposition is not sufficient to answer the question because given $p$, we are unable compute the proportion of elliptic curves for which $p$ is \emph{relevant}.
In particular, we do not know precisely for what proportion of elliptic curves is $p$ a good \emph{ordinary} prime.
A lower bound for this proportion has been computed for small primes and can be found in \cite[Table 1]{KR_submitted}.
From the Hasse interval, one may conclude (roughly) that supersingular elliptic curves should be rare among elliptic curves with good reduction over $\mathbb{F}_p$ (approx. $\frac{1}{2\sqrt{p}}$).
Therefore, as $p$ becomes large, 100\% of the elliptic curves with good reduction at $p$ is of ordinary type (see \cite{Bir68}).
It is also reasonable to expect that as $p$ becomes large, the proportion of rank 0 elliptic curves with $\mu_p(E/\Q)= \lambda_p(E/\Q) =0$ approaches 100\% (see \cite[Conjecture 4.7]{KR21}).
Thus, assuming the rank distribution conjecture, it might be reasonable to expect that as $p$ becomes large, varying over all elliptic curves ordered by height, the proportion with good ordinary reduction at $p$ approaches $\frac{1}{2} \left( 1 - \frac{1}{2\sqrt{p}}\right)\left( 1- \frac{1}{p}\right)$.
\end{Rem*}

\section{\texorpdfstring{Growth of the Selmer group of an elliptic curve in a $\Z/p\Z$-extension}{}}
\label{Appendix}
\par Throughout this section, $p\geq 5$ is a fixed prime number and $E_{/\Q}$ an elliptic curve with good ordinary reduction at $p$ for which the following equivalent conditions are satisfied
\begin{enumerate}[\textup{(}i\textup{)}]
    \item $\mu_p(E/\Q)=0$ and $\lambda_p(E/\Q)=0$,
    \item $\Sel_{p^\infty}(E/\Q_{\cyc})=0$.
\end{enumerate}
Let $L$ be a $\Z/p\Z$-extension of $\Q$.
First, in Proposition \ref{LKR}, we establish a criterion for there to be either a rank-jump or growth in the Shafarevich--Tate group.
The result also explores conditions under which the Selmer group 
\[
\Sel_{p^\infty}(E/\Q)=0\text{ and }\Sel_{p^\infty}(E/L)\neq 0.
\]
This result will then be applied to study a problem in arithmetic statistics which we prove in Theorem \ref{section 6 main thm}.

\subsection{}
Here, we first prove a criterion for non-triviality of the $p$-primary Selmer group. 
\begin{Prop}
\label{LKR}
Let $p$ be an odd prime and $L/\Q$ be a $\Z/p\Z$-extension linearly disjoint from $\Q_{\cyc}$.
In other words, $L\neq \Q_1$, where $\Q_1$ is the first layer in the cyclotomic $\Z_p$-extension.
Let $E_{/\Q}$ be an elliptic curve with good ordinary reduction at $p$ and $\Sigma_L$ be the set of primes $\ell\neq p$ that are ramified in $L$.
Assume that the following conditions are satisfied
\begin{enumerate}[\textup{(}i\textup{)}]
    \item \label{LKR c1} $\rank_{\Z} E(\Q)=0$ and $E(\Q)[p^\infty]=0$,
        \item $\mu_p(E/\Q)=0$ and $\lambda_p(E/\Q)=0$,
    \item there is a prime $\ell\in \Sigma_L$ at which $E$ has good reduction and $p|\#\widetilde{E}(\F_\ell)$,
    \item \label{LKR c4} at each prime $\ell\in \Sigma_L$ at which $E$ has bad reduction, the Kodaira-type of $E_{/\Q_\ell}$ is not $\op{I}_m$ for any integer $m\in \Z_{\geq 1}$,
    \item $\Sha(E/L)[p^\infty]$ is finite.
\end{enumerate}
Then, $\Sel_{p^\infty}(E/\Q)=\Sha(E/\Q)[p^\infty]=0$ and at least one of the following is true
\begin{enumerate}[\textup{(}a\textup{)}]
    \item $\rank_{\Z} E(L)>0$,
    \item $\Sha(E/L)[p^\infty]\neq 0$.
\end{enumerate}
In particular, the Selmer group $\Sel_{p^\infty}(E/L)$ becomes non-zero.
\end{Prop}

\begin{Rem*}
Condition \eqref{LKR c4} above is automatically satisfied when $E$ has good reduction at all primes $\ell\in \Sigma_L$.
In a preprint from 2014, J.~Brau has made an attempt at a comparable result on the growth of Selmer groups (see \cite[Corollary 1.3]{Bra14}).
Our methods differ significantly from those of Brau and do not require the semi-stability hypothesis.
\end{Rem*}

In order to prove the above result, we briefly recall some key properties of the Euler characteristic associated to an elliptic curve.
The reader is referred to \cite{CS00book,Raysujatha1, Raysujatha2} for a more comprehensive discussion of the topic.

Let $E_{/\Q}$ be an elliptic curve and $L/\Q$ a number field extension. Assume that the following conditions are satisfied.
\begin{enumerate}[\textup{(}i\textup{)}]
    \item $\rank_{\Z} E(L)=0$, 
    \item $E$ has good ordinary reduction at $p$, 
    \item $\Sha(E/L)[p^\infty]$ is finite.
\end{enumerate}
Then, the cohomology groups 
$H^i\left(L_{\cyc}/L, \Sel_{p^\infty}(E/L_{\cyc})\right)$ are finite and the \emph{Euler characteristic} $\chi\left(L_{\cyc}/L, E[p^\infty]\right)$ is defined as follows
\[
\chi\left(L_{\cyc}/L, E[p^\infty]\right):=\frac{\# H^0\left(L_{\cyc}/L, \Sel_{p^\infty}(E/L_{\cyc})\right)}{\# H^1\left(L_{\cyc}/L, \Sel_{p^\infty}(E/L_{\cyc})\right)}.
\]
The Euler characteristic is an important invariant associated to the Selmer group, and captures its key Iwasawa theoretic properties.
There is an explicit formula for this invariant, which we now describe.
At each prime $v\nmid p$ of $L$, let $c_v(E/L)$ denote the Tamagawa number at $v$, and let $c_v^{(p)}(E/L)$ be the $p$-part, given by 
$c_v^{(p)}(E/L):=|c_v(E/L)|_p^{-1}$.
Here $\abs{\cdot}_p$ is the absolute value, normalized by $\abs{p}_p=p^{-1}$.
At each prime $v$ of $L$, let $k_v$ be the residue field at $v$, and $\widetilde{E}(k_v)$ be the group of $k_v$-valued points of the reduction of $E$ at $v$.
\begin{Th}\label{boringthm1}
Let $E_{/\Q}$ be an elliptic curve, $p$ an odd prime, and $L/\Q$ be a number field extension.
Assume that the following conditions are satisfied
\begin{enumerate}[\textup{(}i\textup{)}]
    \item $\rank_{\Z} E(L)=0$, 
    \item $E$ has good ordinary reduction at $p$, 
    \item $\Sha(E/L)[p^\infty]$ is finite.
\end{enumerate} Then, the following assertions hold:
\begin{enumerate}[\textup{(}a\textup{)}]
    \item\label{boringthm1 c1} the Euler characteristic $\chi\left(L_{\cyc}/L, E[p^\infty]\right)$ is an integer;
further, it is a power of $p$,
    \item\label{boringthm1 c2} the Euler characteristic is given by the following formula
\begin{equation}
\label{ECF}
\chi\left(L_{\cyc}/L, E[p^\infty]\right)= \frac{\#\Sha(E/L)[p^\infty]\times \prod_v c_v^{(p)}(E/L)\times \prod_{v|p} \left(\#\widetilde{E}(k_v)[p^\infty]\right)^2}{\left(\#E(L)[p^\infty]\right)^2}.
\end{equation}
\end{enumerate}
\end{Th}

\begin{proof}
The assertion \eqref{boringthm1 c1} follows from \cite[Lemma 3.4 and Remark 3.5]{HKR}, and for \eqref{boringthm1 c2}, the reader is referred to \cite[Theorem 3.3]{CS00book}.
\end{proof}

The next result relates the Euler characteristic and Iwasawa invariants.
\begin{Prop}\label{anweshprop}
Let $E_{/L}$ be an elliptic curve satisfying the following conditions:
\begin{enumerate}[\textup{(}i\textup{)}]
    \item $E$ has good ordinary reduction at all primes $v|p$,
    \item $\rank_{\Z} E(L)=0$,
    \item $\Sha(E/L)[p^\infty]$ is finite.
\end{enumerate}
Then, the following are equivalent
\begin{enumerate}[\textup{(}a\textup{)}]
    \item $\mu_p(E/L)=0$ and $\lambda_p(E/L)=0$,
    \item $\chi(L_{\cyc}/L, E[p^\infty])=1$.
\end{enumerate}
\end{Prop}
\begin{proof}
The conditions on the elliptic curve are in place in order for the Euler characteristic to be defined.
The result follows from \cite[Proposition 3.6]{HKR}.
However, the first proof of this result was given in \cite{Raysujatha1}, in a more general context.
\end{proof}

\begin{Lemma}\label{boring lemma 2}
Let $E_{/\Q}$ be an elliptic curve which satisfies the following conditions 
\begin{enumerate}[\textup{(}i\textup{)}]
    \item $E$ has good ordinary reduction at $p$,
    \item $\Sha(E/\Q)[p^\infty]$ is finite,
    \item $\rank_{\Z}E(\Q)=0$,
    \item $E(\Q)[p^\infty]=0$,
    \item
    $\mu_p(E/\Q)=0$ and $\lambda_p(E/\Q)=0$.
\end{enumerate}
Then, we have that $\Sel_{p^\infty}(E/\Q)=\Sha(E/\Q)[p^\infty]=0$.
\end{Lemma}
\begin{proof}
Recall the well-known short exact sequence (see \ref{sesSelmer}),
\[
0\rightarrow E(\Q)\otimes \Q_p/\Z_p\rightarrow \Sel_{p^\infty}(E/\Q)\rightarrow \Sha(E/\Q)[p^\infty]\rightarrow 0.
\]
Since the $E(\Q)$ is finite and $E(\Q)[p^\infty]=0$, it follows from the above that 
\[
\Sel_{p^\infty}(E/\Q)=\Sha(E/\Q)[p^\infty].
\]
It follows from Proposition \ref{anweshprop} that $\chi(\Q_{\cyc}/\Q, E[p^\infty])=1$.
On the other hand, by \eqref{ECF},
\[
\chi\left(\Q_{\cyc}/\Q, E[p^\infty]\right)= \#\Sha(E/\Q)[p^\infty]\times \prod_\ell c_\ell^{(p)}(E/\Q)\times \left(\#\widetilde{E}(\F_p)[p^\infty]\right)^2.
\]
As a result, it is indeed the case the $\Sha(E/\Q)[p^\infty]=0$.
\end{proof}

\begin{Lemma}\label{boringlemma}
Let $p\geq 5$ be a prime, $L/\Q$ be a $\Z/p\Z$ extension, and $\ell\neq p$ be a prime which ramifies in $L$.
Assume that the Kodaira type of $E_{/\Q_\ell}$ is not $\op{I}_m$ for any integer $m\in \Z_{\geq 1}$.
Then, $\prod_{v\mid\ell} c_v^{(p)}(E/L)=1$.
\end{Lemma}
\begin{proof}
Since $\ell\neq p$, and $L/\Q$ is a $\Z/p\Z$-extension, it follows that $\ell$ is tamely ramified in $L$.
Thus, the results for base-change of Tamagawa numbers in \cite[Table 1, pp. 556-557]{Kida} apply.
Fix a prime $v|\ell$ and let $e=e_{L/\Q}(v)$ be the ramification index.
Since it is assumed that $\ell$ is ramified in $L$, it follows that $e=p$.
Since $p\geq 5$, the Tamagawa number $c_v^{(p)}(E/L)\neq 1$ if and only if the Kodaira type of $E_{/L_v}$ is $\op{I}_n$ for an integer $n\in \Z_{\geq 1}$ that is divisible by $p$ (see \cite[p. 448]{Sil09}).
According to \textit{loc. cit.}, the only way this is possible is if the Kodaira type of $E_{/\Q_\ell}$ is $\op{I}_m$ for $m\in \Z_{\geq 1}$.
Indeed, if the Kodaira type of $E_{/\Q_\ell}$ is $\op{I}_m$, then upon base-change to $L_v$, it becomes $\op{I}_{me}=\op{I}_{mp}$.
However, by assumption, this case does not occur.
Therefore, $c_v^{(p)}(E/L)=1$ for all primes $v|\ell$ of $L$.
This completes the proof.
\end{proof}

We now give a proof of Proposition \ref{LKR}.
\begin{proof}[proof of Proposition \ref{LKR}]
According to Lemma \ref{boring lemma 2}, we have that 
\[
\Sel_{p^\infty}(E/\Q)=\Sha(E/\Q)[p^\infty]=0.
\]
Assume by way of contradiction that 
\[
\rank_{\Z} E(L)=0\text{ and }\Sha(E/L)[p^\infty]=0.
\] Since $\rank_{\Z} E(L)=0$, it follows from Theorem \ref{boringthm1} that the Euler characteristic $\chi(L_{\cyc}/L, E[p^\infty])$ is defined and given by the formula 
\begin{equation}\label{yetanotherequation}
\chi\left(L_{\cyc}/L, E[p^\infty]\right)= \frac{\#\Sha(E/L)[p^\infty]\times \prod_v c_v^{(p)}(E/L)\times \prod_{v|p} \left(\#\widetilde{E}(k_v)[p^\infty]\right)^2}{\left(\#E(L)[p^\infty]\right)^2}.
\end{equation}
On the other hand, the Euler characteristic $\chi\left(\Q_{\cyc}/\Q, E[p^\infty]\right)$ is given by the formula
 \[\chi\left(\Q_{\cyc}/\Q, E[p^\infty]\right)= \#\Sha(E/\Q)[p^\infty]\times \prod_\ell c_\ell^{(p)}(E/\Q)\times \left(\#\widetilde{E}(\F_p)[p^\infty]\right)^2.\]
Note that $E(\Q)[p^\infty]$ does not contribute to the above formula since it is assumed to be trivial.
Since it is assumed that $\mu_p(E/\Q)=0$ and $\lambda_p(E/\Q)=0$, it follows from Proposition \ref{anweshprop} that $\chi\left(\Q_{\cyc}/\Q, E[p^\infty]\right)=1$.

We shall use this, and the assumptions on $E$ to show that $\chi\left(L_{\cyc}/L, E[p^\infty]\right)=1$ as well.
Since $\chi\left(\Q_{\cyc}/\Q, E[p^\infty]\right)=1$, it follows that
\[
\# \Sha(E/\Q)[p^{\infty}]=1, \ \prod_\ell c_\ell^{(p)}(E/\Q)=1, \text{ and }\#\widetilde{E}(\F_p)[p^\infty]=1.\]
In order to show that $\chi\left(L_{\cyc}/L, E[p^\infty]\right)=1$, we show that
\[
\# \Sha(E/L)[p^{\infty}]=1, \ \prod_v c_v^{(p)}(E/L)=1, \text{ and }\prod_{v|p}\#\widetilde{E}(k_v)[p^\infty]=1.
\]
By assumption, $\Sha(E/L)[p^\infty]=0$, and hence, $\#\Sha(E/L)[p^\infty]=1$.
It follows from Lemma \ref{boringlemma} that $\prod_{v\nmid p}c_v^{(p)}(E/L)=1$.
If $p$ splits or ramifies in $L$, then $k_v=\F_p$ for all primes $v|p$.
Since $\#\widetilde{E}(\F_p)[p^\infty]=1$, it follows that $\#\widetilde{E}(k_v)[p^\infty]=1$ as well.
On the other hand, suppose $p$ is inert in $L$ and $v|p$ is the only prime above $p$ in $L$.
Then, since $k_v/\F_p$ is a $p$-extension, it follows from \cite[Proposition 1.6.12]{NSW08} that
\[
\#\widetilde{E}(\F_p)[p^\infty]=1\Rightarrow \#\widetilde{E}(k_v)[p^\infty]=1.
\]
Hence, the numerator of \eqref{yetanotherequation} is $1$.
Theorem \ref{boringthm1} asserts the Euler characteristic is an integer, and so, we deduce that $\chi(L_{\cyc}/L, E[p^\infty])=1$.
We deduce from Proposition \ref{anweshprop} that
\[
\mu_p(E/L)=0 \text{ and }\lambda_p(E/L)=0.
\]
In particular, we find that
\[
\lambda_p(E/L)=\lambda_p(E/\Q)=0.
\]
On the other hand, recall that according to Kida's formula 
\[
\lambda_p(E/L)=p\lambda_p(E/\Q)+\sum_{w \in P_1} \left(e_{L/\Q}(w/\ell)-1\right)+\sum_{w \in P_2} 2\left(e_{L/\Q}(w/\ell)-1\right).\]
However, there is a prime $\ell\in \Sigma_L$ for which $E$ has good reduction at $\ell$ and $p| \#\widetilde{E}(\F_\ell)$, we deduce that
\[
\sum_{w \in P_2} 2\left(e_{L/\Q}(w/\ell)-1\right)>0.
\]
Thus, from the above formula shows that
\[
\lambda_p(E/L)>\lambda_p(E/\Q).
\]
In particular,  $\lambda_p(E/L)>0$, this is a contradiction.
Therefore, we have shown that either $\rank_{\Z} E(L)>0$ or $\Sha(E/L)[p^\infty]\neq 0$.
\end{proof}

We illustrate Proposition \ref{LKR} through an example in the case when $p=5$. \newline
\emph{Example:} Let $L/\Q$ be the unique degree $5$ Galois extension contained in $\Q(\mu_{31})$.
The only prime that ramifies in $L$ is $\ell=31$.
Consider the elliptic curve $E: y^2=x^3+42$.
Explicit calculation shows that $E$ satisfies conditions (i)-(iv) of the aforementioned proposition.
Assuming the finiteness of the Shafarevich--Tate group over $L$, we conclude that either there is a rank jump, or an increase in the size of the Shafarevich--Tate group.
It can be checked via standard computations on {\tt Magma} that $E(L)=E(\Q)$.
Therefore, the growth occurs in the Shafarevich--Tate group.

\subsection{}\label{section 4.2}
We now come to an application to arithmetic statistics.
Recall that we have fixed an elliptic curve $E_{/\Q}$ and a prime $p$ for which the aforementioned conditions are satisfied.
Consider the family of $\Z/p\Z$-extensions $L$ of $\Q$ not contained in $\Q_{\cyc}$ and ramified at precisely one prime.
For each prime $q\equiv 1\pmod{p}$, there is exactly one such extension $L_q/\Q$, which is ramified at $q$ (see \ref{well know prop from jr08}).
Note that $L_q$ is contained in $\Q(\mu_q)$.
For $X>0$, let $\pi(X)$ be the prime counting function (i.e., denote the number of primes $q\leq X$) and $\pi'(X)$ be the number of primes $q\leq X$ for which 
\begin{enumerate}[\textup{(}i\textup{)}]
\item $q\equiv 1\pmod{p}$ \emph{and}
\item $\Sel_{p^\infty}(E/L_q)\neq 0$.
\end{enumerate}
Note that since $L_q/\Q$ is a $\Z/p\Z$-extension, we have the following implication
\[
E(\Q)[p^\infty]=0\Rightarrow E(L_q)[p^\infty]=0,
\]
see \cite[Proposition 1.6.12]{NSW08}.
Therefore, the non-vanishing of $\Sel_{p^\infty}(E/L_q)$ is equivalent to at least one of the following conditions being satisfied
\begin{enumerate}[\textup{(}1\textup{)}]
    \item $\rank_{\Z}E(L_q)>0$,
    \item $\Sha(E/L_q)[p^\infty]\neq 0$.
\end{enumerate}
Thus, if the Selmer group becomes non-zero after base-change, then, either there is a rank jump, or the Shafarevich-Tate group witnesses growth.
The main result of this section is Theorem \ref{section 6 main thm}, where it is shown that the set of primes $q\equiv 1\pmod{p}$ for which the above conditions are satisfied is cut out by explicit Chebotarev conditions.
In other words, there is an explicit subset $\mathcal{S}\subset \Gal(\Q(E[p])/\Q)$ such that for any prime $q\neq p$ coprime to the conductor of $E$, 
\[
\Frob_q\in \mathcal{S}\Rightarrow \Sel_{p^\infty}(E/L_q)\neq 0.
\]
Which is to say that if the Frobenius of a prime $q\nmid Np$ lies in the Chebotarev set $\mathcal{S}$, then, the Selmer group becomes non-zero when base changed to $L_q$.
We calculate the size of $\mathcal{S}$ and apply the Chebotarev Density Theorem to obtain a lower bound for $\limsup_{X\rightarrow \infty} \pi'(X)/\pi(X)$.
Said differently, we are able to show there is growth in the Selmer group in $L_q$ for a positive density set of primes $q$.
Let $\bar{\rho}:\Gal(\bar{\Q}/\Q)\rightarrow \GL_2(\F_p)$ be the Galois representation on $E[p]$.
We make the simplifying assumption that $\bar{\rho}$ is surjective.
By the well-known open image theorem of Serre, this assumption is satisfied for all but finitely many primes $p$, as long as $E$ does not have complex multiplication.
Let $\Q(E[p])$ be the Galois extension of $\Q$ which is fixed by $\ker\bar{\rho}$.
We identify $\Gal(\Q(E[p])/\Q)$ with its image under $\bar{\rho}$, which according to our assumption is all of $\GL_2(\F_p)$.
Let $N=N_E$ be the conductor of $E$.
If $q$ is a prime which is coprime to $N p$, then $q$ is unramified in $\Q(E[p])$.
Set $a_q(E):=q+1-\# \widetilde{E}(\F_q)$; then the characteristic polynomial of $\bar{\rho}(\Frob_q)$ is $x^2-a_q x+q$.
Let $\mathcal{S}$ consist of elements $\sigma\in\Gal(\Q(E[p])/\Q)$ such that 
\[
\trace\bar{\rho}(\sigma)=2\text{ and }\det\bar{\rho}(\sigma)=1.\]
We arrive at the following useful criterion for $p$ to divide $\#\widetilde{E}(\F_q)$.
\begin{Lemma}
\label{lemma useful criterion}
Let $q\nmid N p$ be a prime. 
Then, the following conditions are equivalent
\begin{enumerate}[\textup{(}i\textup{)}]
    \item $q\equiv 1\pmod{p}$ and $p$ divides $\#\widetilde{E}(\F_q)$, 
    \item $\Frob_q\in \mathcal{S}$.
\end{enumerate}
\end{Lemma}
\begin{proof}
Since $\det\bar{\rho}(\Frob_q)=q$, we find that $\det\bar{\rho}(\Frob_q)=1$ if and only if $q\equiv  1\pmod{p}$.
Assume that these equivalent conditions are satisfied.
Note that $p$ divides $\# \widetilde{E}(\F_q)$ if and only if $q+1-\trace\left(\bar{\rho}(\Frob_q)\right)=0$.
Since $q\equiv 1\pmod{p}$, it follows that $p$ divides $\# \widetilde{E}(\F_q)$ if and only if $\trace\left(\bar{\rho}(\Frob_q)\right)=2$.
\end{proof}
\begin{Lemma}\label{S cardinality}
The cardinality of $\mathcal{S}$ is $p^2$.
\end{Lemma}
\begin{proof}
Since $\bar{\rho}$ is assumed to be surjective, we identify $\mathcal{S}$ with the set of all matrices in $\GL_2(\F_p)$ with trace $2$ and determinant $1$.
These matrices are all of the form $g=\mtx{a}{b}{c}{2-a}$, where $a(2-a)-bc=1$.
We count the number of of such matrices.
Rewrite the equation as $bc=-1+a(2-a)=-(a-1)^2$.
For each choice of $a$ such that $a\neq 1$, the number of solutions is $(p-1)$.
For $a=1$, either $b=0$ or $c=0$, or both.
Thus the number of solutions for $a=1$ is $2p-1$.
Putting it all together,
\[
\# \mathcal{S}=(p-1)^2+(2p-1)=p^2.
\]
\end{proof}
We now prove Theorem \ref{Anwesh's thm}, which is the main result of this section.
\begin{Th}
\label{section 6 main thm}
Let $p\geq 5$ be a fixed prime and $E_{/\Q}$ an elliptic curve for which the following conditions are satisfied
\begin{enumerate}[\textup{(}i\textup{)}]
    \item $E$ has good ordinary reduction at $p$, 
    \item $\rank_{\Z} E(\Q)=0$ and $E(\Q)[p^\infty]=0$,
    \item $\mu_p(E/\Q)=0$ and $\lambda_p(E/\Q)=0$,
    \item the image of the residual representation $\bar{\rho}$ is surjective.
\end{enumerate}
For $X>0$, let $\pi'(X)$ be the number of primes $q\equiv 1\pmod{p}$ for which the following equivalent conditions are satisfied
\begin{enumerate}[\textup{(}i\textup{)}]
    \item $\Sel_{p^\infty}(E/L_q)\neq 0$,
    \item Either, $\rank_{\Z} E(L_q)>0$ or $\Sha(E/L_q)[p^\infty]\neq 0$ (or both conditions are satisfied).
\end{enumerate}
Then,
\[
\limsup_{X\rightarrow \infty} \frac{\pi'(X)}{\pi(X)}\geq \frac{p}{(p-1)^2(p+1)}.
\]
\end{Th}

\begin{proof}
Let $q$ be a prime such that $q\nmid N p$.
It follows from Lemma \ref{lemma useful criterion} that if $\Frob_q\in \mathcal{S}$, then $q\equiv 1\pmod{p}$ and $p|\#E(\F_q)$.
Note that $\Sigma_{L_q}=\{q\}$ and $E$ has good reduction at $q$.
Since $p|\#E(\F_q)$, it follows from that Proposition \ref{LKR} that $\Sel_{p^\infty}(E/L_q)\neq 0$.
Therefore, by the Chebotarev Density Theorem, and Lemma $\ref{S cardinality}$,
\[
\limsup_{X\rightarrow\infty} \frac{\pi'(X)}{\pi(X)}\geq \frac{\# \mathcal{S}}{\#\left(\GL_2(\F_p)\right)}=\frac{p^2}{(p^2-1)(p^2-p)}=\frac{p}{(p-1)^2(p+1)}.
\]
\end{proof}

\section{Growth of Shafarevich--Tate groups in cyclic extensions}
\label{section: growth of Sha}
In this section, we study the growth of the Shafarevich--Tate group.
First, when $p=2$ we prove an effective version of Matsuno's theorem (see Theorem \ref{prop B mat09}).

When $p\neq 2$, it was shown by Clark and Sharif (see \cite[Theorem~3]{CS10}) that there exists a degree-$p$ extension over $\Q$, not necessarily Galois such that the $p$-rank of the Shafarevich--Tate group becomes arbitrarily large.
We study the possibility of improving this result to Galois degree-$p$ extensions.

\begin{Defi}
Let $p$ be any prime.
Fix an elliptic curve $E_{/\Q}$ with good reduction at $p$.
Let $K/\Q$ be a cyclic degree-$p$ extension.
Define the set $T_{E/K}$ to be the set of primes in $K$ above $\ell(\neq p)$ satisfying either of the following properties
\begin{enumerate}[\textup{(}i\textup{)}]
\item $\ell$ is a prime of good reduction of $E$ which is ramified in $K/\Q$ and $E(\Q_{\ell})$ contains an element of order $p$.
\item $\ell$ is a prime of split multiplicative reduction which is inert in $K/\Q$ and the Tamagawa number $c_{\ell}$ is divisible by $p$.
\end{enumerate}
\end{Defi} 

\begin{Defi}
Let $G$ be an abelian group.
Define the \emph{$p$-rank} of $G$ as
\[
\rank_p (G) = \rank_p (G[p]) := \dim_{\mathbb{F}_p}\left(G[p]\right).
\] 
\end{Defi}
The following lemma plays a key role in answering questions pertaining to both sections.
\begin{Lemma}
\label{lemma: 2-rank of selmer}
With notation as above,
\[
\rank_{p} \Sel_p(E/K) \geq \# T_{E/K} - 4.
\]
\end{Lemma}
\begin{proof}
See \cite[Proposition 4.3]{Mat09}.
\end{proof}

\subsection{Large 2-rank of the Shafarevich--Tate group in quadratic extensions}
Given an elliptic curve $E_{/\Q}$, it is known that the 2-part of the Shafarevich--Tate group becomes arbitrarily large over \emph{some} quadratic extension.
More precisely, 
\begin{Th}
\label{prop B mat09}
Given an elliptic curve $E_{/\Q}$ and a non-negative integer $n$, there exists a quadratic number field $K/\Q$ such that $\dim_{\mathbb{F}_2}\Sha(E/K)[2]\geq n$.
\end{Th}

\begin{proof}
See {\cite[Proposition B]{Mat09}}.
\end{proof}

A reasonable question to ask is whether the above result can be made effective, i.e., 
\begin{Ques*}
Given an elliptic curve $E_{/\Q}$ and a fixed non-negative integer $n$,
varying over all quadratic number fields ordered by conductor, what is the minimal conductor of a number field such that one can guarantee $\Sha(E/K)[2]\geq n$?
\end{Ques*}

\subsubsection{Reviewing the proof of Matsuno's construction}
In this section, we briefly review the proof of Matsuno's theorem, for details we refer the reader to the original article \cite{Mat09}.
Fix an elliptic curve $E_{/\Q}$ of conductor $N=N_E$.
Let $K/\Q$ be a quadratic extension. 
Let $S$ be a finite set of primes in $\Q$ containing precisely the prime number 2, the primes of bad reduction of $E$, and the archimedean primes.

Let $\ell_1, \ldots, \ell_k$ be odd rational primes that are coprime to $N$ and split completely in $\Q(E[2])/\Q$.
Recall that $\Q(E[2])/\Q$ is a Galois extension with Galois group isomorphic to either $\Z/2\Z$ or $\Z/3\Z$ or $S_3$.
By the Chebotarev Density Theorem there is a positive proportion of primes that split completely in $\Q(E[2])/\Q$.
Having picked the primes $\ell_1, \ldots, \ell_k$, it is clear that there exists a quadratic extension $K/\Q$ such that the chosen primes ramify.
However, more is true.
Results of J.-L.~Waldspurger (see \cite[Theorem in Section~0]{Bump}) and V.~Kolyvagin (see \cite{Kol89}) guarantee the existence of $K/\Q$ such that the chosen primes ramify \emph{and} the Mordell--Weil rank of $E^\prime(\Q)$ is 0 where $E^\prime$ is the quadratic twist of $E$ corresponding to $K$. 
It follows that
\[
\rank_{\Z}(E(K)) = \rank_{\Z}(E(\Q)) + \rank_{\Z} (E^\prime(\Q)) = \rank_{\Z} (E(\Q)). 
\]
Observe that the primes $\ell_1, \ldots, \ell_k$ are primes of good ordinary reduction that lie in $T_{E/K}$.
Therefore, it follows from Lemma \ref{lemma: 2-rank of selmer} that
\[
\rank_{2} \Sel_2(E/K) \geq k - 4.
\]
Using the Kummer sequence, we see that
\begin{equation}
{\begin{split}
\rank_2 \Sha(E/K)[2] &\geq \rank_{2}\Sel_2(E/K) - \rank_{\Z}(E/K) -2\\
& \geq k - 6 - \rank_{\Z} (E(\Q)).
\end{split}}
\end{equation}

Even though the results of Waldspurger and Kolyvagin guarantee that there exists a quadratic extension $K/\Q$ with the desired properties, it is not easy to determine all the primes which ramify in $K$.
Using the proof of Matsuno as our inspiration, combined with more recent results of K.~Ono we will prove an effective version of the theorem in Section~\ref{S: unconditional Matsuno}.
Assuming Goldfeld's Conjecture, we prove an effective version using simpler arguments in Section~\ref{S: conditional effective}. 

\subsubsection{Effective version conditional on Goldfeld's conjecture}
\label{S: conditional effective}
Let us begin by reminding the reader of Goldfeld's Conjecture.
This conjecture predicts that given an elliptic curve, $50\%$ of the quadratic twists have rank 0 and $50\%$ of the quadratic twists have rank 1.
Therefore, if Goldfeld's Conjecture is true then a $100\%$ of the time, the Mordell--Weil rank of $E^\prime(\Q)$ is either 0 or 1, where $E^\prime$ is the quadratic twist of $E$ (corresponding to a quadratic field $K$).
Since
\[
\rank_{\Z}(E(K)) = \rank_{\Z}(E(\Q)) + \rank_{\Z} (E^\prime(\Q)), 
\]
for a $100\%$ of the time $\rank_{\Z}(E(K))=\rank_{\Z}(E(\Q))$ or $\rank_{\Z}(E(K))=\rank_{\Z}(E(\Q))+1$.
Suppose that the primes $\ell_1, \ldots, \ell_k$ are chosen as in Matsuno's theorem.
These are primes of good ordinary reduction that lie in $T_{E/K}$.
Therefore, it follows from Lemma \ref{lemma: 2-rank of selmer} that
\[
\rank_{2} \Sel_2(E/K) \geq k - 4.
\]
Using the Kummer sequence, we see that for a $100\%$ of the quadratic fields $K$,
\begin{equation}
\label{2-sha over K in terms of 2-sha over Q}
{\begin{split}
\rank_2 \Sha(E/K)[2] &\geq \rank_{2}\Sel_2(E/K) - \rank_{\Z}(E(K)) -2\\
& \geq k - 7 - \rank_{\Z} (E(\Q)).
\end{split}}
\end{equation}

Given $E_{/\Q}$ of conductor $N$ and $n\in \Z_{\geq 0}$, we want to find an imaginary quadratic field $K/\Q$ with minimal conductor $\mathfrak{f}_{K}$ such that we can guarantee $\rank_2 \Sha(E/K)[2]\geq n$.
Let $\mathcal{P} = \{\ell_1, \ldots, \ell_k\}$ be a set of (distinct) rational primes not dividing $N$ with the additional property that they split completely in $\Q(E[2])/\Q$ and that precisely the primes in $\mathcal{P}$ ramify in $K$. 
From \eqref{2-sha over K in terms of 2-sha over Q}, we need that
\[
\rank_2 \Sha(E/K)[2] \geq k - 7 - \rank_{\Z} (E(\Q)) \geq n.
\]
Equivalently,
\[
k \geq n + \rank_{\Z} (E(\Q)) + 7.
\]
Note that this inequality ensures that $k$ can take all but finitely many values.
Therefore, in view of Goldfeld's Conjecture there exists a non-negative integer $\epsilon_E$ such that the set $\mathcal{P}$ contains $k+\epsilon_{E}$ (which we still call $k$ by abuse of notation) many primes instead.
To answer our question, we need to carefully pick the distinct primes $\ell_1, \ldots, \ell_k$.
Given any integer $M$, the number of distinct prime factors denoted by $\omega(M)$ is asymptotically $\log \log M$.
Recall that the Prime Number Theorem asserts that the average gap between consecutive primes among the first $x$ many primes is $\log x$.
Since $\ell_1$ is a prime of good reduction, $\ell_1\nmid N$.
The Prime Number Theorem implies that $\ell_1 \sim \log \left( \omega (N)\right) \sim \log_{(3)} (N)$.
However, we need to do more. 
We require that $\ell_1$ splits completely in $\Q(E[2])/\Q$.
Using the Chebotarev Density Theorem, we can conclude that $\ell_1 \sim c \cdot \log \left( \omega (N)\right)$, where $c$ is either $2$, $3$ or $6$ depending on the degree of the Galois extension $\Q(E[2])/\Q$.
Of course, asymptotically $\ell_1 \sim \log  \left( \omega (N)\right)$.
Next, $\ell_2 \nmid N\ell_1$ and splits completely in $\Q(E[2])/\Q$. 
Note that the number of distinct prime divisors of $N\ell_1 \sim \omega(N)+ 1$.
Using the same argument, we see that $\ell_2 \sim \log\left( \omega(N) + 1\right)$.
Continuing this process, we have
\begin{align}
\mathfrak{f} = \prod_{i=1}^k \ell_i &\sim \prod_{i=1}^k\left( \log \left(\omega(N) + (i-1)\right)\right) \\
&\sim \frac{(\log (\omega(N)-1+k))^{\omega(N)+2+k}}{\exp \li(\omega(N)-1+k)} \\
&\sim \frac{(\log n)^{n+c}}{\exp \li(n)} \label{last expressions}
\end{align} 
where $\li(x) := \int_0^x \frac{dt}{\log t}$ is the logarithmic integral and $c$ is a constant depending on $E$.
Let us explain the above estimates in greater detail. Setting
\[
P(k):= \prod_{i=1}^k \log (\omega(N)+(i-1)),
\] we wish to estimate the sum
\[
\log(P(k))= \sum_{i=1}^k \log \left(\log (\omega(N)+(i-1))\right).
\]

Using Abel's partial summation formula (see, for example \cite[Theorem 4.2]{Apostol}) with $f(t) = \log(\log(\omega(N) - 1 + t))$, $a(n) = 1$, and $A(x) = \sum_{n\leq x}a(n)$),  one obtains 
\begin{align*}
\log P(k) &= \sum_{i=1}^k \log\log(\omega(N) - 1 + i)\\
&= k \log\log(\omega(N) - 1 + k) - \int_0^k \frac{\lfloor x \rfloor}{(\omega(N) - 1 + x)\log(\omega(N) - 1 + x)}\, dx \\
&= k \log\log(\omega(N) - 1 + k) - \int_0^k \frac{x}{(\omega(N) - 1 + x)\log(\omega(N) - 1 + x)}\, dx \\
& \hspace{2em}+ O\left( \int_0^k \frac{dx}{(\omega(N) - 1 + x)\log(\omega(N) - 1 + x)}\right) \\
&= k \log\log(\omega(N) - 1 + k) - \left( \li(\omega(N) - 1 + x) - (\omega(N) - 1)\log\log(\omega(N) - 1 + x) \right)\Big|_0^k \\
& \hspace{2em} + O\left( \int_0^k \frac{dx}{(\omega(N) - 1 + x)\log(\omega(N) - 1 + x)}\right) \\
&= k \log\log(\omega(N) - 1 + k)  - \li(\omega(N) - 1 + k) + (\omega(N) - 1)\log\log(\omega(N) - 1 + k) \\
& \hspace{2em} + O(\log\log(\omega(N) - 1 + k)) \\
&= (k + \omega(N) + 1)\log\log(\omega(N) - 1 + k) - \li(\omega(N) - 1 + k) + O(\log\log(\omega(N) - 1 + k)).
\end{align*}
This tells us that 
\begin{align*}
P(k) &= \exp\left((k + \omega(N) + 1)\log\log(\omega(N) - 1 + k) - \li(\omega(N) - 1 + k) + O(\log\log(\omega(N) - 1 + k))\right)\\
&= \frac{\exp((k + \omega(N) + 1)\log\log(\omega(N) - 1 + k))}{\exp( \li(\omega(N) - 1 + k))}\left( \exp(O(\log\log(\omega(N) - 1 + k)))\right) \\
&= \frac{\exp((k + \omega(N) + 1)\log\log(\omega(N) - 1 + k))}{\exp( \li(\omega(N) - 1 + k))}\left( O(\log(\omega(N) - 1 + k))\right) \\
&= \frac{\log(\omega(N) - 1 + k)^{k + \omega(N) + 2}}{\exp( \li(\omega(N) - 1 + k))}. 
\end{align*}

For obtaining \eqref{last expressions}, we note that for the given $n$, the difference between $n$ and $k$ depends on the rank of $E$, and further $\omega(N)$ depends on $E$.
Thus, we can rewrite $\omega(N)+2+k$ as $n+c$ where $c$ is a constant depending only on the elliptic curve $E$.

We have therefore proven the following result.
\begin{Th}
\label{thm: effective matsuno}
Suppose that Goldfeld's Conjecture is true.
Given an elliptic curve $E_{/\Q}$ and a positive integer $n$, there exists a quadratic extension $K/\Q$ with conductor $\mathfrak{f}_K \sim \frac{(\log n)^{n+c}}{\exp \li (n)}$ such that $\rank_2 \Sha(E/K)[2] \geq n$.
Here $c$ is an explicit constant depending only on $E$.
\end{Th}

Using a result of Ono we can prove an unconditional statement which we now explain.

\subsubsection{An unconditional effective version} 
\label{S: unconditional Matsuno}
Given an elliptic curve $E_{/\Q}$ of conductor $N$ and a positive integer $n$, we want to find an imaginary quadratic field $K/\Q$ with minimal conductor $\mathfrak{f}_{K}$ such that we can guarantee $\rank_2 \Sha(E/K)[2]\geq n$.

For this section, we consider elliptic curves $E_{/\Q}$ which have no exceptional primes $p$, i.e., no primes $p$ such that the $\mod p$ Galois representation attached to $E$ is non-surjective.
This condition is a mild one: W.~Duke has shown in \cite[Theorem~1]{Duk97} that almost all elliptic curves defined over $\Q$ have no exceptional primes.
 
For such elliptic curves $E_{/\Q}$, \cite[Theorem~1]{Ono01} asserts that there exists a fundamental discriminant $D_E$ such that the twisted curve $E^{d}$ has rank $0$ where $d = D_E\ell_1\cdots \ell_{k}$ for some even integer $k$ and where the primes $\ell_i$ are chosen from a set (of primes) with density $1$.
Indeed, the aforementioned theorem asserts (more generally) that there is a set of primes (call it $S$) with positive Frobenius density\footnote{i.e., there exists a finite Galois extension $L/\Q$ such that for all but finitely many primes in this set, these primes represent a fixed Frobenius conjugacy class in $\Gal(L/\Q)$} such that $E^{d}$ has rank $0$ where $d = D_E\ell_1\cdots \ell_{k}$ for some even integer $k$ and the primes $\ell_i$ are chosen from $S$.
But it follows from \cite[proof of Theorem~2.2]{Ono01} that the field $L/\Q$ in the definition of the Frobenius density is the field which contains the coefficients of the newform associated to $E$.
Since the elliptic curves we are working with are defined over $\Q$, it follows from the Modularity Theorem that the newform will have integral coefficients i.e., $L = \Q$. 
Since there is only one Frobenius class in $\Q$, namely the trivial class, \cite[Theorem~1]{Ono01} says that the density of the set $S$ is $1$.

Let $K$ be the quadratic field associated with the twist $d$.
Since $E^d$ has rank 0, we know that $\rank_{\Z}(E(K))=\rank_{\Z}(E(\Q))$.
It follows that
\[
\rank_2 \Sha(E/K)[2] \geq k - 6 - \rank_{\Z} (E(\Q)).
\]
If $K$ is a quadratic field arising from Ono's theorem such that $\rank_2 \Sha(E/K)[2]>n$, it is required that
\[
k \geq n + \rank_{\Z} (E(\Q)) + 6.
\]
To answer our question, we pick the distinct primes $\ell_1, \ldots, \ell_k$ using the exact same process as before.
We have that
\[
\mathfrak{f} = D_E\prod_{i=1}^k \ell_i \sim \prod_{i=1}^k\left( \log \left(\omega(N) + (i-1)\right)\right) \sim \frac{(\log (\omega(N)+k))^{\omega(N)+k}}{\exp \li(\omega(N)+k)} \sim \frac{(\log n)^{n+c}}{\exp \li(n)}\] 
where $\li(x) := \int_0^x \frac{dt}{\log t}$ is the logarithmic integral and $c$ is a constant depending on $E$.

We have therefore proven the following result.
\begin{Th}
\label{thm: effective matsuno unconditional}
Given an elliptic curve $E_{/\Q}$ with no exceptional primes and a positive integer $n$, there exists a quadratic extension $K/\Q$ with conductor $\mathfrak{f}_K \sim \frac{(\log n)^{n+c}}{\exp \li (n)}$ such that $\rank_2 \Sha(E/K)[2] \geq n$.
Here $c$ is an explicit constant depending only on $E$.
\end{Th}

\subsection{\texorpdfstring{Arbitrarily large $\Sha$ in $\Z/p\Z$-Extensions}{}}
\label{subsection: arbitrarily large Sha in odd p case}
Let $E_{/\Q}$ be a fixed rank 0 elliptic curve of conductor $N$.
Let $p$ be a fixed odd prime.
Let $F$ be any number field, we have the obvious short exact sequence
\[
0 \rightarrow E(F)/pE(F) \rightarrow \Sel_p(E/F) \rightarrow  \Sha(E/F)[p] \rightarrow 0.
\]
It follows that
\begin{equation}
\label{p-rank equation}
\rank_{p} \Sel_p(E/F) = \rank_{\Z}(E(F)) + \rank_{p} E(F)[p] + \rank_p\Sha(E/F)[p].
\end{equation}
When $F/\Q$ is a cyclic degree-$p$ Galois extension, we know from Lemma \ref{lemma: 2-rank of selmer} that
\[
\rank_p \Sel_p(E/F) \geq \# T_{E/F} -4.
\]
Given an integer $n$, there exists a number field $F_{(n)}$ such that (see \cite[Theorem 1.2]{Ces17}) 
\[
\rank_{p}\Sel_{p}(E/F_{(n)}) \geq \# T_{E/F} -4\geq n.
\]
Denote the conductor of $F_{(n)}$ by $\mathfrak{f}(F_{(n)})$.
Varying over \emph{all} $\Z/p\Z$-extensions of $\Q$, there are infinitely many number fields $L/\Q$ such that $T_{E/L} \supseteq T_{E/F_{(n)}}$, i.e., $\mathfrak{f}(F_{(n)})\mid \mathfrak{f}(L)$.
For each such $L$,
\[
\rank_{p}\Sel_{p}(E/L) \geq n.
\]
Recall the conjecture of David--Fearnley--Kisilevsky from Section \ref{section: preliminaries}.
If $p\geq 7$, it predicts the boundedness of the set
\[
N_{E,p}(X) := \{L/\Q \textrm{ cyclic of degree } p: \mathfrak{f}(L)< X \textrm{ and } \rank_{\Z}(E/L) > \rank_{\Z}(E/\Q)\}.
\]
If this conjecture is true, then varying over \emph{all} $\Z/p\Z$-extensions of $\Q$, there are only finitely many cyclic $p$-extensions $L/\Q$ in which there is a rank jump upon base-change.
Therefore, given an integer $n$, one can find an integer $M = M(n)$ and a number field $L/\Q$ such that 
\begin{enumerate}[\textup{(}i\textup{)}]
\item $\mathfrak{f}(L)> M(n)$,
\item $\mathfrak{f}(F_{(n)})\mid \mathfrak{f}(L)$,
\item $\rank_{\Z}(E(L)) = \rank_{\Z}(E(\Q))$.
\end{enumerate}
Thus, given $n$, there exists a cyclic extension $L/\Q$ of degree $p\geq 7$ such that \eqref{p-rank equation} becomes
\[
\rank_{p} \Sel_p(E/L) = \rank_{\Z}(E(\Q)) + \rank_p E(L)[p] + \rank_p\Sha(E/L)[p] \geq n.
\]
Since $\rank_{\Z}(E(\Q))$ is independent of $L$ and $\rank_p E(L)[p]$ is at most 2, it means that given $n$, there exists a $\Z/p\Z$-extension $L/\Q$ such that $\rank_p\Sha(E/L)[p]\geq n$.

We feel that the full force of the conjecture of David--Fearnley--Kisilevsky is required.
In particular, we do not see if the result of Mazur--Rubin is sufficient.
This is because it is not obvious to us as to why even if there are infinitely many number fields $\mathcal{L}/\Q$ with $\rank_{\Z}(E(\mathcal{L})) = \rank_{\Z}(E(\mathcal{\Q}))$, there should be any (of these) satisfying $\mathfrak{f}({F_{(n)}})\mid \mathfrak{f}(\mathcal{L})$.

\section{Tables}
\label{section: Tables}

\subsection{}In Table \ref{table: values for positive prop}, we compute the following expression
\[
\left(1 - \frac{1}{p}\right)
\left( 1 - \prod_{q \equiv 1 (\textrm{mod }{p})}
\left( \frac{q-1}{2q^2} + 1 - \frac{1}{q}\right) \right).
\]
for $3\leq p < 50$ and $q \leq 179,424,673$ (i.e., the first 10 million primes).

\begin{longtable}{| c | c || c | c |} 
\caption{Values for expression \eqref{eqn: expression}}
\label{table: values for positive prop}\\
\hline 
Primes & Value for Expression & Primes & Value for Expression
\\ [0.5ex]
\hline 
3 & 0.293282 & 23 & 0.0404621\\
5 & 0.189719 & 29 & 0.0303331\\
7 & 0.121798 & 31 & 0.0198836\\
11 & 0.0866316 & 37 & 0.0197385\\
13 & 0.0647841 & 41 & 0.019329\\
17 & 0.0453478 & 43 & 0.0165664\\
19 & 0.0342828 & 47 & 0.0141439\\
\hline
\end{longtable}

\subsection{} In Table \ref{table: enemy primes}, we record the proportion of enemy primes for CM elliptic curves with conductor $<100$ and $p<50$.
Here, ``-'' indicates that for a given elliptic curve, $p$ is an \emph{irrelevant} prime.
Since $p=47$ is an \emph{irrelevant prime} for the four elliptic curves of interest, we have excluded it from our table.
In the first row, the value of $\frac{p}{2(p+1)(p-1)^2}$ is recorded.
Note that the reduction type depends only on the isogeny class.
The same is true for the $\lambda$-invariant of the $p$-primary Selmer group. 
However, it is possible that one or more of the curves in a given isogeny class has positive $\mu$-invariant, but the others do not (see also \cite[Conjecture 1.11]{Gre99}).
To keep the tables succinct, since the Iwasawa invariants are the same for all curves in the isogeny class, they are clubbed together.

The data in the table is obtained from the code available \href{https://github.com/kundudeb/kundudeb.github.io/blob/master/SAGE_code_BKR.pdf}{here}.

\begin{longtable}{| c | r | r | r | r | r | r | r |}
\caption{Proportion of enemy primes} 
\label{table: enemy primes}\\
\hline
E & 5 & 7 & 11 & 13 & 17 & 19 \\ [0.5ex]
\hline
$p/2(p+1)(p-1)^2$ & 0.002604 & 0.012153 & 0.004583 & 0.003224 & 0.001845 & 0.001466 \\ [0.5ex]
\hline
\bf{27a} & 0.020833 & 0.013881 & - & 0.003471 & - & 0.001533 \\
\hline
\bf{36a} & - & 0.013880 & - & 0.003471 & - & 0.001538 \\
\hline
\bf{49a} & -	& -	& 0.004989	& -	& -& 	-\\
\hline
\bf{64a} & 0.031260 & - & - & 0.003478 & 0.001950 & - \\
\hline
\end{longtable}

\newpage
\begin{longtable}{| c | r | r | r | r | r | r |}
\caption*{Proportion of enemy primes continued} 
\label{table: enemy primes cont}\\
\hline
E & 23 & 29 & 31 & 37 & 41 & 43\\ [0.5ex]
\hline
$p/2(p+1)(p-1)^2$ & 0.000990 & 0.000616 & 0.000538 & 0.000376 & 0.000305 & 0.000277 \\ [0.5ex]
\hline
\bf{27a} & - & - & 0.000551 & 	0.000424 & - & 	0.000279	\\
\hline
\bf{36a} & - & - & 0.000550 & 0.000422 & - & 0.000282 \\
\hline
\bf{49a} & 0.001028 &	0.000638 &	- &	0.000423 &	- &	0.000282 \\
\hline
\bf{64a} & - & 0.000640 & - & 0.000426 & 0.000312 & - \\
\hline
\end{longtable}

\section*{Acknowledgements}
LB and DK thank Rahul Arora, Christopher Keyes, Debasis Kundu, Allysa Lumley, Jackson Morrow, Kumar Murty, Mohammed Sadek, and Frank Thorne for helpful discussions during the preparation of this article.
LB and DK thank Henri Darmon for his continued support.
LB acknowledges the support of the CRM-ISM Postdoctoral Fellowship.
DK acknowledges the support of the PIMS Postdoctoral Fellowship.
AR's research is supported by the CRM Simons postdoctoral fellowship.
This work was initiated by LB and DK during the thematic semester ``Number Theory -- Cohomology in Arithmetic'' at Centre de Recherches Math{\'e}matiques (CRM) in Fall 2020.
LB and DK thank the CRM for the hospitality and generous supports.
AR thanks Ravi Ramakrishna, R. Sujatha, and Tom Weston for their support and guidance, and thanks Larry Washington for helpful conversations.
We thank Henri Darmon, Chantal David, Antonio Lei, Robert Lemke Oliver, Jackson Morrow, Ross Paterson, Ravi Ramakrishna, and Larry Washington for their comments on an earlier draft of this paper.
We thank the referee for the timely review and for the suggestions that led to various improvements in the exposition.

\bibliographystyle{alpha}
\bibliography{references}

\newcommand{\etalchar}[1]{$^{#1}$}
\begin{thebibliography}{DEvH{\etalchar{+}}21}

\bibitem[Apo76]{Apostol}
Tom~M. Apostol.
\newblock {\em Introduction to analytic number theory}.
\newblock Undergraduate Texts in Mathematics. Springer-Verlag, New
  York-Heidelberg, 1976.

\bibitem[BCH{\etalchar{+}}66]{BCHIS66}
Armand Borel, Sarvadaman Chowla, Carl~S Herz, Kenkichi Iwasawa, and Jean-Pierre
  Serre.
\newblock {\em Seminar on Complex Multiplication: Seminar Held at the Institute
  for Advanced Study, Princeton, NY, 1957-58}, volume~21 of {\em Lecture Notes
  in Mathematics}.
\newblock Springer, 1966.

\bibitem[BFH90]{Bump}
Daniel Bump, Solomon Friedberg, and Jeffrey Hoffstein.
\newblock Nonvanishing theorems for $l$-functions of modular.
\newblock {\em Invent. math}, 102:543--618, 1990.

\bibitem[Bir68]{Bir68}
Bryan~J Birch.
\newblock How the number of points of an elliptic curve over a fixed prime
  field varies.
\newblock {\em J. London Math. Soc.}, 1(1):57--60, 1968.

\bibitem[BN16]{bruin2016criterion}
Peter Bruin and Filip Najman.
\newblock A criterion to rule out torsion groups for elliptic curves over
  number fields.
\newblock {\em Res. Number Theory}, 2(1):3, 2016.

\bibitem[Bra14]{Bra14}
Julio Brau.
\newblock Selmer groups of elliptic curves in degree $p$ extensions, 2014.
\newblock arXiv:1401.3304.

\bibitem[BS13]{BS13}
Manjul Bhargava and Arul Shankar.
\newblock The average size of the 5-{S}elmer group of elliptic curves is 6, and
  the average rank is less than 1, 2013.
\newblock arXiv:1312.7859.

\bibitem[BS15a]{BS15_quartic}
Manjul Bhargava and Arul Shankar.
\newblock Binary quartic forms having bounded invariants, and the boundedness
  of the average rank of elliptic curves.
\newblock {\em Ann. Math.}, pages 191--242, 2015.

\bibitem[BS15b]{BS15_cubic}
Manjul Bhargava and Arul Shankar.
\newblock Ternary cubic forms having bounded invariants, and the existence of a
  positive proportion of elliptic curves having rank 0.
\newblock {\em Ann. Math.}, pages 587--621, 2015.

\bibitem[{\v{C}}es17]{Ces17}
K{\k{e}}stutis {\v{C}}esnavi{\v{c}}ius.
\newblock $p$-{S}elmer growth in extensions of degree $p$.
\newblock {\em J. London Math. Soc.}, 95(3):833--852, 2017.

\bibitem[Coj04]{Coj04}
Alina~Carmen Cojocaru.
\newblock Questions about the reductions modulo primes of an elliptic curve.
\newblock In {\em Proc. 7th Meeting of the Canadian Number Theory Association
  (Montreal, 2002)}, pages 61--79. Citeseer, 2004.

\bibitem[CS00]{CS00book}
John Coates and Ramdorai Sujatha.
\newblock {\em Galois cohomology of elliptic curves}.
\newblock Narosa, 2000.

\bibitem[CS10]{CS10}
Pete~L Clark and Shahed Sharif.
\newblock Period, index and potential, {III}.
\newblock {\em Algebra Number Theory}, 4(2):151--174, 2010.

\bibitem[CS21]{CS20}
J.~E. Cremona and M.~Sadek.
\newblock Local and global densities for {W}eierstrass models of elliptic
  curves.
\newblock {\em accepted for publication in Mathematical Research Letters},
  2021.
\newblock preprint, arXiv:2003.08454.

\bibitem[DEvH{\etalchar{+}}21]{derickx2021sporadic}
Maarten Derickx, Anastassia Etropolski, Mark van Hoeij, Jackson~S Morrow, and
  David Zureick-Brown.
\newblock Sporadic cubic torsion.
\newblock {\em Algebra Number Theory}, 15(7):1837--1864, 2021.

\bibitem[DFK07]{DFK07}
Chantal David, Jack Fearnley, and Hershy Kisilevsky.
\newblock Vanishing of ${L}$-functions of elliptic curves over number fields.
\newblock {\em Ranks of elliptic curves and random matrix theory}, (341):247,
  2007.

\bibitem[DN19]{derickx2018torsion}
Maarten Derickx and Filip Najman.
\newblock Torsion of elliptic curves over cyclic cubic fields.
\newblock {\em Math. Comput.}, 88(319):2443--2459, 2019.

\bibitem[Dok07]{Dok07}
Tim Dokchitser.
\newblock Ranks of elliptic curves in cubic extensions.
\newblock {\em Acta Arithmetica}, 4(126):357--360, 2007.

\bibitem[Duk97]{Duk97}
William Duke.
\newblock Elliptic curves with no exceptional primes.
\newblock {\em Proceedings of the Acad{\'e}mie des sciences. Series 1,
  Math{\'e}matics}, 325(8):813--818, 1997.

\bibitem[GFP20]{GFP20}
Natalia Garcia-Fritz and Hector Pasten.
\newblock Towards {H}ilbert's tenth problem for rings of integers through
  {I}wasawa theory and {H}eegner points.
\newblock {\em Math. Ann.}, 377(3):989--1013, 2020.

\bibitem[GJN20]{GJN20}
Enrique Gonz{\'a}lez-Jim{\'e}nez and Filip Najman.
\newblock Growth of torsion groups of elliptic curves upon base change.
\newblock {\em Math. Comput.}, 89(323):1457--1485, 2020.

\bibitem[Gre99]{Gre99}
R~Greenberg.
\newblock Iwasawa theory for elliptic curves.
\newblock In {\em Arithmetic theory of elliptic curves (Cetraro, 1997)}, volume
  1716, pages 51--144. Springer, 1999.

\bibitem[HKR21]{HKR}
Jeffrey Hatley, Debanjana Kundu, and Anwesh Ray.
\newblock Statistics for anticyclotomic {I}wasawa invariants of elliptic
  curves, 2021.
\newblock arXiv:2106.01517.

\bibitem[HL19]{HL19}
Jeffrey Hatley and Antonio Lei.
\newblock Arithmetic properties of signed {S}elmer groups at non-ordinary
  primes.
\newblock {\em Ann. Inst. Fourier}, 69(3):1259--1294, 2019.

\bibitem[HM99]{HM99}
Yoshitaka Hachimori and Kazuo Matsuno.
\newblock An analogue of {K}ida's formula for the {S}elmer groups of elliptic
  curves.
\newblock {\em J. Alg. Geom.}, 8(3):581--601, 1999.

\bibitem[JKL11]{jeon2011families}
Daeyeol Jeon, Chang Kim, and Yoonjin Lee.
\newblock Families of elliptic curves over cubic number fields with prescribed
  torsion subgroups.
\newblock {\em Math. Comput.}, 80(273):579--591, 2011.

\bibitem[JKS04]{jeon2004torsion}
Daeyeol Jeon, Chang~Heon Kim, and Andreas Schweizer.
\newblock On the torsion of elliptic curves over cubic number fields.
\newblock {\em Acta Arithmetica}, 113:291--301, 2004.

\bibitem[JR08]{JR08}
John~W Jones and David~P Roberts.
\newblock Number fields ramified at one prime.
\newblock In {\em International Algorithmic Number Theory Symposium}, pages
  226--239. Springer, 2008.

\bibitem[Kam92]{Kam92}
Sheldon Kamienny.
\newblock Torsion points on elliptic curves and $q$-coefficients of modular
  forms.
\newblock {\em Invent. math.}, 109(1):221--229, 1992.

\bibitem[Kat04]{Kat04}
Kazuya Kato.
\newblock $p$-adic {H}odge theory and values of zeta functions of modular
  forms.
\newblock {\em Ast{\'e}risque}, 295:117--290, 2004.

\bibitem[Kid03]{Kida}
Masanari Kida.
\newblock Variation of the reduction type of elliptic curves under small base
  change with wild ramification.
\newblock {\em Central European J. Math.}, 1(4):510--560, 2003.

\bibitem[KM88]{KM88}
Monsur~A Kenku and Fumiyuki Momose.
\newblock Torsion points on elliptic curves defined over quadratic fields.
\newblock {\em Nagoya Mathematical Journal}, 109:125--149, 1988.

\bibitem[Kol89]{Kol89}
Viktor~Alexandrovich Kolyvagin.
\newblock Finiteness of and for a subclass of {W}eil curves.
\newblock {\em Mathematics of the USSR-Izvestiya}, 32(3):523, 1989.

\bibitem[KR21a]{KR21}
Debanjana Kundu and Anwesh Ray.
\newblock Statistics for {I}wasawa invariants of elliptic curves.
\newblock {\em Trans. Am. Math. Soc.}, 374:7945--7965, 2021.
\newblock DOI: https://doi.org/10.1090/tran/8478.

\bibitem[KR21b]{KR_submitted}
Debanjana Kundu and Anwesh Ray.
\newblock Statistics for {I}wasawa invariants of elliptic curves, {II}, 2021.
\newblock arXiv:2106.12095.

\bibitem[LR22]{lozanocm}
\'{A}lvaro Lozano-Robledo.
\newblock Galois representations attached to elliptic curves with complex
  multiplication.
\newblock {\em Algebra Number Theory}, 16(4):777--837, 2022.

\bibitem[Mat09]{Mat09}
Kazuo Matsuno.
\newblock Elliptic curves with large {T}ate--{S}hafarevich groups over a number
  field.
\newblock {\em Math. Res. Lett.}, 16(3):449--461, 2009.

\bibitem[Maz72]{Maz72}
Barry Mazur.
\newblock Rational points of abelian varieties with values in towers of number
  fields.
\newblock {\em Invent. math.}, 18(3-4):183--266, 1972.

\bibitem[Maz77]{Maz77}
Barry Mazur.
\newblock Modular curves and the {E}isenstein ideal.
\newblock {\em Publ. Math. IHES}, 47(1):33--186, 1977.

\bibitem[Maz78]{Maz78}
Barry Mazur.
\newblock Rational isogenies of prime degree.
\newblock {\em Invent. math.}, 44(2):129--162, 1978.

\bibitem[Mer96]{Mer96}
Lo{\"\i}c Merel.
\newblock Bornes pour la torsion des courbes elliptiques sur les corps de
  nombres.
\newblock {\em Invent. math.}, 124(1-3):437--449, 1996.

\bibitem[MP22]{MP21}
Adam Morgan and Ross Paterson.
\newblock On 2-{S}elmer groups of twists after quadratic extension.
\newblock {\em J. London Math. Soc.}, 105(2):1110--1166, 2022.

\bibitem[MR19]{MR19_arithmetic_conjectures}
Barry Mazur and Karl Rubin.
\newblock Arithmetic conjectures suggested by the statistical behavior of
  modular symbols.
\newblock {\em arXiv preprint arXiv:1910.12798}, 2019.

\bibitem[MRL18]{MR18}
Barry Mazur, Karl Rubin, and Michael Larsen.
\newblock Diophantine stability.
\newblock {\em Am. J. Math.}, 140(3):571--616, 2018.

\bibitem[MSM16]{MM16}
Guillermo Mantilla-Soler and Marina Monsurr{\`o}.
\newblock The shape of $\mathbb{Z}/\ell\mathbb{Z}$-number fields.
\newblock {\em Ramanujan J.}, 3(39):451--463, 2016.

\bibitem[Naj16]{najman2016torsion}
Filip Najman.
\newblock Torsion of rational elliptic curves over cubic fields and sporadic
  points on ${X}_1 (n)$.
\newblock {\em Math. Res. Lett.}, 23(1):245--272, 2016.

\bibitem[NSW13]{NSW08}
J{\"u}rgen Neukirch, Alexander Schmidt, and Kay Wingberg.
\newblock {\em Cohomology of number fields}, volume 323.
\newblock Springer, 2013.

\bibitem[Ono01]{Ono01}
Ken Ono.
\newblock Nonvanishing of quadratic twists of modular ${L}$-functions and
  applications to elliptic curves.
\newblock {\em J. reine angew. math.}, 533:81--97, 2001.

\bibitem[Poo18]{Poo18}
Bjorn Poonen.
\newblock Heuristics for the arithmetic of elliptic curves.
\newblock In {\em Proceedings of the International Congress of
  Mathematicians-Rio de}, volume~2, pages 399--414. World Scientific, 2018.

\bibitem[PPVW19]{PPVW19}
Jennifer Park, Bjorn Poonen, John Voight, and Melanie~Matchett Wood.
\newblock A heuristic for boundedness of ranks of elliptic curves.
\newblock {\em J. European Math. Soc.}, 21(9):2859--2903, 2019.

\bibitem[RS19]{Raysujatha1}
Anwesh Ray and Ramdorai Sujatha.
\newblock Euler characteristics and their congruences in the positive rank
  setting.
\newblock {\em Canad. Math. Bull.}, pages 1--19, 2019.

\bibitem[RS20]{Raysujatha2}
Anwesh Ray and R~Sujatha.
\newblock Euler characteristics and their congruences for multisigned {S}elmer
  groups.
\newblock {\em Canadian J. Mathematics}, pages 1--24, 2020.

\bibitem[{Sag}20]{sagemath}
{Sage Developers}.
\newblock {\em {S}ageMath, the {S}age {M}athematics {S}oftware {S}ystem
  ({V}ersion 9.2)}, 2020.
\newblock {\tt https://www.sagemath.org}.

\bibitem[Ser74]{Ser74}
Jean-Pierre Serre.
\newblock Divisibilit{\'e} de certaines fonctions arithm{\'e}tiques.
\newblock {\em S{\'e}minaire Delange-Pisot-Poitou. Th{\'e}orie des nombres},
  16(1):1--28, 1974.

\bibitem[Sil09]{Sil09}
Joseph~H Silverman.
\newblock {\em The arithmetic of elliptic curves}, volume 106 of {\em Graduate
  Texts in Mathematics}.
\newblock Springer, 2009.

\bibitem[ST68]{ST68}
Jean-Pierre Serre and John Tate.
\newblock Good reduction of abelian varieties.
\newblock {\em Ann. Math.}, pages 492--517, 1968.

\bibitem[Wan15]{wang2015torsion}
Jian Wang.
\newblock {\em On the torsion structure of elliptic curves over cubic number
  fields}.
\newblock PhD thesis, 2015.

\bibitem[Was97]{Was97}
Lawrence~C Washington.
\newblock {\em Introduction to cyclotomic fields}, volume~83 of {\em Graduate
  Texts in Mathematics}.
\newblock Springer, 1997.

\bibitem[Was08]{Was08}
Lawrence~C Washington.
\newblock {\em Elliptic curves: number theory and cryptography}.
\newblock CRC press, 2008.

\end{thebibliography}
\end{document}